\documentclass[final,3p]{elsarticle}

\usepackage{natbib}

\usepackage{amsmath,amsthm, amssymb,amsfonts}
\usepackage{colonequals}
\usepackage{graphicx}
\usepackage{color,epic}
\usepackage{verbatim}
\usepackage{accents}
\usepackage{bbold}
\usepackage{amsmath}
\usepackage{amsfonts}
\usepackage{amssymb}%
\usepackage{soul}
\usepackage{comment}

\providecommand{\U}[1]{\protect\rule{.1in}{.1in}}

\def\lowerbar{\underaccent{\bar}}

\DeclareMathOperator* \argmin {arg \, min}
\numberwithin{equation}{section}

\newtheorem{theorem}{Theorem}
\newtheorem{lem}{Lemma}
\newtheorem{cor}{Corollary}
\newtheorem*{cors}{Corollary}%

\theoremstyle{remark}
\newtheorem{rem}{Remark}
\newtheorem{example}{Example}

\specialcomment{arxiv}{}{}%
\specialcomment{jet}{}{}%
\includecomment{arxiv}
\excludecomment{jet}

\journal{Journal of Economic Theory}
\begin{document}

\begin{frontmatter}

\title{Stochastic Stability of Monotone Economies in Regenerative Environments}

\author[sf]{Sergey Foss}
\ead{s.foss@hw.ac.uk}
\address[sf]{Maxwell Institute and Heriot-Watt University,
Edinburgh and Sobolev Institute of Mathematics and Novosibirsk State University.}

\author[vs]{Vsevolod Shneer}
\ead{v.shneer@hw.ac.uk}
\address[vs]{Maxwell Institute and Heriot-Watt University.}

\author[jt]{Jonathan P.\ Thomas}
\ead{jonathan.thomas@ed.ac.uk}
\address[jt]{School of Economics, University of Edinburgh.}

\author[tsw]{Tim Worrall\corref{cor1}}
\ead{tim.worrall@ed.ac.uk}
\address[tsw]{School of Economics, University of Edinburgh.}
\cortext[cor1]{Corresponding author.}

\begin{abstract}
We introduce and analyze a new class of monotone stochastic recursions in a
regenerative environment which is essentially broader than that of Markov
chains. We prove stability theorems and apply our results {to three canonical
models in recursive economics}, generalizing some known stability results to
the cases when driving sequences are not independent and identically
distributed.
\end{abstract}


\end{frontmatter}

\section{Introduction}

\label{section1}

{This paper develops results on stochastic stability, in particular, uniform convergence to a unique stationary distribution, for a class of monotone stochastic recursions where the exogenous stochastic driving process is regenerative. A regenerative stochastic process is, loosely, a process that has independent and identically distributed (i.i.d.)\ cycles. We apply our results to three important workhorse models in macroeconomics. The Bewley-Imrohoroglu-Huggett-Aiyagari precautionary savings model
of \citet{Bewley86}, \citet{Imrohoroglu92}, \citet{Huggett93} and
\citet{Aiyagari94}, the one-sector stochastic optimal growth model of
\citet{BrockMirman72}, and the risk-sharing under limited commitment model of
\citet{Kocherlakota96}. In each of these examples, we are able to demonstrate uniqueness and stability
results under less restrictive assumptions than in existing literature.}

{To illustrate the applicability of our approach, consider a typical problem in economic dynamics that can be solved recursively using a  Bellman equation of the form
\begin{equation}
V(x, z) = \sup_{x^\prime \in \Gamma(x,z)} u(x,z, x^\prime) + \beta \int V(x^\prime, z^\prime) Q(z, dz^\prime). \label{eq:be}
\end{equation}
In this equation $x$ is an endogenous state variable, $z$ is an exogenous shock, $u$ is the per-period payoff function, $\Gamma$ is the constraint set, $Q$ is the transition function for the shock and $V$ is the value function. Variables indicated by a $^\prime$ are the next period values. Stochastic  dynamic  programming problems of this type are discussed extensively in \citet{Stokey-Lucas89}.\footnote{This formulation of a control problem is sometimes described as being of the Euler class \citep{Miao14}.} When there is a unique solution to the Bellman equation, it can be described by a policy function $x^\prime = f(x, z)$.}

{The policy function from dynamic problems of the type described in~\eqref{eq:be} are examples of a
\emph{stochastic recursive sequence (SRS)\/}, or \emph{stochastic recursion\/}
of the form
\begin{equation}
X_{t+1}=f(X_{t},Z_{t})\quad\mbox{a.s.}, \label{eq:srs}%
\end{equation}
where $\{Z_{t}\}$ is a stochastic process with $Z_{t}\in\mathcal{Z}$,
$X\in\mathcal{X}$ is the state variable of economic interest and
$f\mathpunct{:}{\cal X}\times\mathcal{Z}\rightarrow\mathcal{X}$ is an
appropriately measurable function. The process $\{Z_{t}\}$ is known as the
\emph{driving sequence\/} of the stochastic recursion. For a given $X_{0}$ and
given (random) values of $Z_{0},\ldots,Z_{t-1}$, the system~\eqref{eq:srs}
generates a (random) value of $X_{t}$.}

{It is well-known that a stochastic recursive sequence is more general than a Markov chain \citep[see, e.g.,][]{BoFo}.\footnote{We follow the terminology of \citet{MTw} and use the term Markov Chain to refer to any discrete-time Markov process (DTMP) whether the state space is finite, countable or continuous.} In particular, under extremely general
conditions on the state space $\mathcal{X}$
(see Section 2.1 for details), any time-homogeneous Markov chain
 (equivalently, discrete-time Markov process, DTMP)
may be represented as an SRS~\eqref{eq:srs} with independent and identically
distributed (i.i.d.)\ driving elements $Z_{0},Z_{1},\ldots$, whereas, the stochastic recursion allows $\{Z_n\}$ to be dependent, for example, it could itself be a Markov chain.}

{\citet{Stachurski} gives a number of examples of stochastic recursions in economics including threshold models and random mutations to best responses in a co-ordination game. Other examples include linear models, such as $X_{t+1}=a_t X_t + b_t$ where $Z_t=(a_t, b_t)$ is a random vector \citep{Horst}. The focus of our applications will however, be on recursions generated from dynamic programming problems of the type in equation~\eqref{eq:be}.}

{The question we address in this paper is whether there exist a unique stationary distribution for $X$ when the driving process is regenerative. The answer to this question depends} on the spaces $\mathcal{X}$ and
$\mathcal{Z}$, the function $f$ and the nature of the driving sequence. In
this paper we are concerned with the case where the function $f$ is
monotone \emph{increasing\/} in $X$ and where $Z$ is a \emph{regenerative
process\/}. We make appropriate assumptions on $\mathcal{X}$ and $\mathcal{Z}$
that are specified below. Loosely, a stochastic process is regenerative if it
can be split into independent and identically distributed (i.i.d.)\ cycles.
That is, if there exists a subsequence of (random) dates such that the process
has the same probabilistic behavior between any two consecutive dates in the
subsequence. The cycle lengths (lengths of time intervals between these dates)
may also be random, in general, with the only requirement that they have a
finite mean value. As an example, consider a finite-state time-homogeneous
Markov chain with a single closed class of communicating states. If the chain
starts in some state~$z_{0}$, then the subsequence of dates corresponds to the
dates at which the chain revisits state~$z_{0}$. Between each of these dates
the chain has the same probabilistic behavior.\footnote{An i.i.d.\ process is
one that is regenerative at every date.} The class of regenerative processes
is large and includes not only ergodic Markov chains, but also renewal
processes, Brownian motion, waiting times in general queues and so
on.\footnote{We are not the first to consider regenerative processes in the
economics literature. For example, \citet{Kamihigashi15} consider perfect
simulation of a stochastic recursion of the form $X_{t+1}=f(X_{t},\xi
_{t})\mathbb{1}\{X_{t}\geq x\}+\epsilon_{t}\mathbb{1}\{X_{t}<x\}$ where
$\mathcal{X}=[a,b]$, $x\in(a,b)$, $f$ is increasing in $X$ and $\{\xi_{t}\}$
and $\{\epsilon_{t}\}$ are i.i.d. The process regenerates for values $X_{t}%
<x$. This process arises in models of industry dynamics with entry and exit
\citep[see][]{Hopenhayn92}. It is a Markov process, but it is not monotone
unless the distribution of $f(x,\xi)$ stochastically dominates the
distribution of $\epsilon$.}

Before explaining our approach in more detail, we outline three traditional
approaches that are used to address stability and uniqueness issues for SRS of
the type described by equation~\eqref{eq:srs}.
First, when $\{Z_{t}\}$ is i.i.d., the process for $X_{t}$ is Markov and
standard existence and convergence results for discrete-time Markov processes can be
applied. For example, when $f$ is monotone in the first argument, it is
well-known that there is convergence {to a unique invariant distribution} if a
mixing or splitting condition holds \citep[see, e.g.,][]{Dubins, BhMa,
Stokey-Lucas89, HopenhaynPrescott92}.\footnote{\citet{Stokey-Lucas89} use the
Feller property, which is a continuity requirement, together with monotonicity
and a mixing condition to derive the results. \citet{HopenhaynPrescott92}
develop an existence result using monotonicity alone, and combined with a
mixing condition, establish that uniqueness and stability follow.}

Second, stability results are also known in a more general setting where the
driving sequence $\{Z_{t}\}$ is stationary or even asymptotically stationary
(this literature originated with \citet{Loynes}, see, e.g., \citet{BoFo} and
references therein). By \emph{stationarity\/} we mean \emph{stationarity in
the strong sense\/}, that is, for any finite $k$, the distribution of a
finite-dimensional vector $(Z_{t},\ldots,Z_{t+k})$ does not depend on $t$. The
most basic result is that if the state space for the $X$'s is partially
ordered and possesses a least element, say 0, and if SRS $X_{t+1}%
=f(X_{t},Z_{t})$ starts from the bottom point $X_{0}=0$, with $f$ monotone
increasing in the first argument, then the distribution of $X_{t}$ is monotone
increasing in $t$ and, given that the sequence is \emph{tight\/}%
,\footnote{Tightness in this {context} means that for any $\varepsilon>0$
there exists $K_{\varepsilon}$ such that ${\mathbf{P}}(X_{t}\geq
K_{\varepsilon})\leq\varepsilon$ for all $t$.} it converges to a limit which
is the \emph{minimal\/} stationary solution to recursion (1.1). In general,
there may be many solutions, and for the minimal solution to be unique, one
has to require additional assumptions, such as, the existence of \emph{renovating events\/} \citep[for details
see, e.g.,][]{Foss83, Brandt85}. These results seem to have been relatively
little used in the economics literature although in \citet{Bewley86} it is
assumed that there is a Markov driving sequence for shocks that starts from a
stationary state.

A third situation where results are known is considered by
\citet[][Chapter~9]{Stokey-Lucas89} and \citet{HopenhaynPrescott92}. If
$\{Z_{t}\}$ is itself a Markov chain, or equivalently an SRS of the form
$Z_{t}=g(Z_{t-1},\varepsilon_{t-1})$ with i.i.d.\ $\{\varepsilon_{t}\}$, then
$Y_{t}=(X_{t},Z_{t})$ is a time-homogeneous Markov chain, equivalently, an SRS
of the form $Y_{t+1}=F(Y_{t},\varepsilon_{t})\colonequals(f(X_{t}%
,Z_{t}),g(Z_{t},\varepsilon_{t}))$. Then, provided a mixing condition is
satisfied, one can use the monotone convergence approach to establish
convergence of the extended Markov chain $Y_{t}$. This is the approach
generally used in the economics literature. There are however, three
disadvantages to this approach. First, in order to apply monotone convergence
results, it is required that function $g$ is increasing in the first argument.
That is, it is required that the driving process is itself monotone
(positively correlated). Whilst this may be natural in many economic contexts,
it may be restrictive in others.\footnote{For example, if the states that the
driving process represents have no natural ordering, there may be no
reordering of states such that the process is monotone. We give further
examples below.} Second, to apply monotone convergence results, it is required
that function $f$ is monotone (increasing) in both arguments, not just the
first argument. This can be problematic in situations where the SRS is derived
as a policy function of a dynamic programming problem. In this case,
establishing monotonicity in the second argument may require extra
restrictions on preferences and technology. This is the case in the one sector
stochastic optimal growth model with correlated shocks that is studied by
\citet{DonaldsonMehra} and others; see section 3.2. Third, the fact that the
state space for the extended state variable, $\mathcal{X}\times\mathcal{Z}$,
is of a larger dimension, may create additional technical difficulties and
establishing that the mixing condition is satisfied may become less
straightforward.

In this paper we exploit the i.i.d.\ cycle property of regenerative processes.
We use this property to construct a Markov process defined at the regeneration
times driven by an i.i.d.\ random variable. Together with an analogue of the
monotone mixing or \emph{splitting condition\/} of
\citet[][condition~(1.2)]{BhMa2} this can be used to establish convergence to
a unique stationary distribution.

We develop our approach in a simple scenario with a compact and completely
ordered state space $\mathcal{X}$ (which may be taken to be $[a,b],$ $a,b\in%
\mathbb{R}
$, $a<b,$ without loss of generality). In the case where the driving sequence
is i.i.d., the splitting condition says that (we focus here on the
i.i.d.\ case for simplicity of notation and explanations), for some
$c\in\lbrack a,b]$, there is a finite time $N$ such that for the Markov chain
$X_{t}^{(b)}$ that starts from the maximal state $X_{0}^{(b)}=b$ at time zero
(with any $Z_{0}$), the probability ${\mathbf{P}}(X_{N}^{(b)}\leq c)>0$ and,
second, for the Markov chain $X_{t}^{(a)}$ that starts from the minimal state
$X_{0}^{(a)}=a$ at time zero (with any $Z_{0}$), the probability ${\mathbf{P}%
}(X_{N}^{(a)}\geq c)>0$. In Section~2.1 we reproduce a result of \citet{BhMa}
for the case where $f$ is monotone increasing, the driving sequence is
i.i.d.\ and the splitting condition holds (Theorem~1) that shows there is
exponentially fast convergence to a unique stationary distribution. Theorem~2
in Section 2.2 extends this result to allow for a regenerative driving
sequence. A corollary to this theorem (Corollary~1) is provided in Section~2.3
that considers the important special case where the driving sequence is itself
an aperiodic Markov chain with a positive atom. For such regenerative driving
sequences, our approach generalizes the standard result whilst avoiding the
disadvantages mentioned above. In particular, we establish convergence to a
unique stationary distribution without needing to assume the driving process
is itself monotone or that the function $f$ is increasing in the second
argument. In addition, our convergence applies directly to the state space of
interest, $\mathcal{X}$, {and can be extended to the joint distribution on the state} space $\mathcal{X}%
\times\mathcal{Z}$.



The paper is organized as follows. In Section~\ref{MM}, we describe the model
and provide our main results. First, we describe regenerative processes. Next,
we review the results of~\citet{BhMa} for an i.i.d.\ driving sequence. Then,
we present the main results showing that if a mixing condition similar to that
given in \citet{BhMa} are satisfied between the dates when the driving
sequence regenerates, then stability holds. {Section~\ref{sec:eg} presents the
three economic applications of our main result to an income fluctuation problem with savings (Section~\ref{subsec:huggett}), stochastic optimal growth (Section~\ref{subsec:growth}) and risk sharing with limited commitment (Section~\ref{subsec:risksharing}). The proofs of the main and other subsidiary proofs are put in the Appendix.}

\section{The Main Model}

\label{MM}

In this section, we {outline the} main properties of discrete-time
regenerative processes, provide our lead example of regeneration for
Markov chains, and introduce our main model, which is a stochastic recursive
sequence with a regenerative driver.

Let~$Z_{t},t=0,1,\ldots$ be a (one-sided) regenerative sequence on a general
measurable space~$(\mathcal{Z},\mathcal{B}_{\mathcal{Z}})$. The sequence is
\emph{regenerative\/} if there exists an increasing sequence of integer-valued
random variables (times)~$0=T_{-1}\leq T_{0}<T_{1}<T_{2}<\ldots$ such that,
for~$\tau_{n}=T_{n}-T_{n-1},n\geq0$, the vectors
\begin{equation}
\{\tau_{n},Z_{T_{n-1}},\ldots,Z_{T_{n}-1}\} \label{regene}%
\end{equation}
are independent for~$n\geq0$ and identically distributed for~$n\geq1$. A
random vector \eqref{regene} is called a \emph{cycle\/} with \emph{cycle
length\/} $\tau_{n}$ and with $\{Z_{T_{n-1}},\ldots,Z_{T_{n}-1}%
\}$ the sequence of ``shocks'' over the cycle starting at the regenerative
time $T_{n-1}$ and up to the period before the next regenerative
time, i.e., $T_{n}-1$.

Furthermore, we assume that
\begin{equation}
\mathbf{E}\tau_{1}<\infty. \label{eq:finitemean}%
\end{equation}
It is known \citep[see, e.g.,][]{Asmussen} that if, in addition, regenerative
times are \emph{aperiodic\/},
\begin{equation}
G.C.D.\{n:\ {\mathbf{P}}(\tau_{1}=n)>0\}=1, \label{eq:aperiodicity}%
\end{equation}
 then $Z_{t}$ has a unique stationary distribution,\footnote{Throughout we use the term unique stationary distribution and in our context this equivalent to a unique limiting distribution for any initial value of the process. Other terms used for stationary distribution are invariant and steady-state distribution.} say $\pi$, and
converges to it in the total variation norm:
\[
\sup_{B\in\mathcal{B_{Z}}}|{\mathbf{P}}(Z_{t}\in B)-\pi(B)|\rightarrow
0,\quad\mbox{a.s.}\quad t\rightarrow\infty.
\]

The main aim of the paper is to study the behavior of a recursive sequence
\begin{equation}
X_{t+1}=f\left(  X_{t},Z_{t}\right), \quad
t=0,1,\ldots ,  \label{eq:generalsetting}%
\end{equation}
that starts from $X_0=x\in{\cal X}$, assuming that
\begin{itemize}
\item {} the function $f$ is measurable and is \emph{monotone\/} in the first
argument, with respect to some
ordering;

\item {} sequence $\{Z_{t}\}$ is regenerative and satisfies conditions
\eqref{eq:finitemean}-\eqref{eq:aperiodicity}.\footnote{A minor and natural extension is to the case where the recursive
sequence is $X_{t+1}=f\left(  X_{t},\xi_{t}^{Z_{t}}\right)$, $f$  monotone in its first argument, in which
$\left\{  \xi_{t}^{z}\right\}_{z\in\mathcal{Z},-\infty<t<\infty}$ are a
family of mutually independent random variables. With the assumption that each
$z\in\mathcal{Z}$, $\left\{  \xi_{t}^{z}\right\}  _{t\geq1}$ are i.i.d.\ with a
common distribution, it can be shown that our main theorem holds for this more general driving process.}
\end{itemize}

\begin{example}\label{iidexample} {The simplest possible example of a regenerative process is when $\{Z_{t}\}$ is an i.i.d.\ process. In this case $T_n=n$ and $\tau_n=1$ for $n\geq1$. All cycles are of length one.}
\end{example}

\begin{example}
\label{markovexample0} {In many economic applications the driving process is modeled as a time-homogenous, irreducible and aperiodic Markov chain $\{Z_{t}\}$ taking values in a
finite state space $Z$. In this case we can pick any particular state $z_{0}$ and then every time the process returns to $z_{0}$, a new
sequence is formed from the states occurring until $z_{0}$ is
visited again. The regeneration times $T_{0}<T_{1}<T_{2}\ldots$ are the
hitting times of $z_{0}$. By the Markov property, these sequences
and their length are independent and identically distributed. Similarly, the hitting times are aperiodic and \eqref{eq:finitemean}-\eqref{eq:aperiodicity} are satisfied.}
\end{example}

\begin{example}
\label{markovexample} {Example~\ref{markovexample0} is easily generalized to a positive recurrent time-homogeneous Markov chain with a
general state space $(\mathcal{Z},\mathcal{B_{Z}})$ that has a positive atom. A Markov chain has a positive atom if there is a point
$z_{0}\in\mathcal{Z}$ such that, for any $z\in\mathcal{Z}$,
\[
T_{1}^{z}=\min\{t\ :Z_{t}=z_{0}\ |\ Z_{0}=z\}<\infty\quad\mbox{a.s.}
\]
and
\[
{\mathbf{E}}T_{1}^{z_{0}}<\infty.
\]
Again the regeneration times $T_{0}<T_{1}<T_{2}\ldots$ are the
hitting times of $z_{0}$. By the Markov property, these sequences
and their length are independent and identically distributed.
Provided these hitting times are additionally assumed to be aperiodic, then \eqref{eq:finitemean}-\eqref{eq:aperiodicity} are satisfied.}
\end{example}

{Most of the known results on stability for stochastic recursions are for the case where the driving process is i.i.d.\ or the driving process is Markov \emph{and\/} monotone increasing. Our extension is to provide similar stability results for any regenerative process including Markov processes that are not monotonic. The applications we consider in Sections~\ref{subsec:growth} and~\ref{subsec:risksharing} are with Markov driving processes as in Example~\ref{markovexample0} and the application considered in Section~\ref{subsec:huggett} is with a driving process defined on a general state space as in Example~\ref{markovexample}. Similarly, models where a potentially non-Markov process drives an
agent's environment, but it periodically returns to some initial state, can be
incorporated into our framework, as in the next example. }

\begin{example}
\label{ex:search}
A worker who has just entered the unemployment pool
at $t=0$ receives unemployment benefit $b$ until
successfully matched with a firm, thereafter receiving wages $w_{t}$ until a
separation occurs, whereupon the worker returns to the initial unemployment
state (i.e., as at date $0)$. Wages and the matching and separation
hazards evolve jointly according to a general stochastic process. Formally let
$\{E_{t},y_{t}\}$ represent the process where $E_{t}\in\left\{
0,1\right\}  $ represents employment status ($0$ for unemployed, $1$ for
employed) and $y_{t}$ is income at time~$t$ ($y_{t}=b$
when $E_{t}=0$), and $E_{0}=0$. Then, $\{E_{t},y_{t}%
\}$ is a regenerative process with regenerative times~$\{T_{j}%
\}$ given by each time the worker transitions from employment to
unemployment: $T_{-1}=T_{0}=0$, $T_{1} =\min\left\{
t>0:E_{t-1}=0,E_{t}=1\right\}  $, the first time the worker returns to
unemployment, and likewise for each 
$j=2,\ldots$, let %
\[
T_{j}=\min\{t>T_{j-1}:E_{t-1}=0,E_{t}=1\}.
\]
Then, provided the mean return time to the initial state is finite and
the return times are mutually independent and have an aperiodic distribution (e.g., if transition probabilities are positive
at each date), assumptions \eqref{eq:finitemean}-\eqref{eq:aperiodicity} are
satisfied.
\end{example}


{In the rest of this section, we first consider the standard case with an i.i.d.\ driving process. In Section~\ref{subsec:regen} we provide the result of our main theorem for a regenerative driving process. In Section~\ref{subsec:markov} we specialize our result to the case where the driving process is a Markov chain with a countable state space and a positive recurrent atom. Finally, in Section~\ref{subsec:discuss} we discuss our results in relation to some of the existing literature on monotone economies.}

\subsection{I.i.d.\ driving sequence}\label{subsec:iid}

We start with a particular case when~$Z_{t}\ $is i.i.d. We revisit some
results from \citet{BhMa} (see also \citet{Dubins}).

The relation between time-homogeneous Markov chains (with a general measurable
state space $(\mathcal{X},\mathcal{B}_{\mathcal{X}})$) and recursions
\eqref{eq:generalsetting} with i.i.d.\ drivers is well-understood
(see, e.g., \cite{Kifer, BoFo}): if the sigma-algebra $\mathcal{B}_{\mathcal{X}}$ is
countably generated, then a Markov chain may be represented as a stochastic
recursion \eqref{eq:generalsetting} with an i.i.d.\ driving sequence
$\{Z_{t}\}$. In particular, any real-valued or vector-valued time-homogeneous
Markov chain may be represented as a stochastic recursion \eqref{eq:generalsetting}.

In what follows, we restrict our attention to real-valued $X_{t}$ and,
moreover, assume that
\begin{equation}
\mbox{the state space}\quad\mathcal{X}\quad\mbox{is the closed interval}\quad
\lbrack a,b]. \label{eq:zero-one}%
\end{equation}
We define the \emph{uniform, or Kolmogorov distance\/} between probability distributions on
the real line as
\begin{equation}
d(F,G)=\sup_{x}|F(x)-G(x)|. \label{uni1}%
\end{equation}
Here $F(x)=F(-\infty,x]$ and $G(x)=G(-\infty,x]$ are the distribution
functions. Let $F(x-)=F(-\infty,x)$ and $G(x-)=G(-\infty,x)$.\footnote{Note that convergence in the uniform distance is weaker than convergence in the total variation norm.} Then,
by the right-continuity of distribution functions,
\begin{equation}
d(F,G)=\sup_{x}|F(x-)-G(x-)|\equiv\sup_{x}\max\left(
|F(x-)-G(x-)|,|F(x)-G(x)|\right)  . \label{uni2}%
\end{equation}

Next, we assume the function $f$ to be \emph{monotone increasing\/} in the
first argument: for each $z\in\mathcal{Z}$ and for each $a\leq x_{1}\leq
x_{2}\leq b$,
\[
f(x_{1},z)\leq f(x_{2},z).
\]
We write for short
\[
{\mathbf{P}}^{(x)}(\cdot)={\mathbf{P}}(\cdot\ |\ X_{0}=x).
\]
\begin{jet}
We also denote by $F_{t}^{(x)}$ the distribution function of the random
variable $X_{t}$ if $X_{0}=x$ (and more generally denote by $F_{t}^{(\mu_0)}$
the distribution function of $X_{t}$ if $X_{0}$ has distribution $\mu_0$). Our first Theorem reproduces a result of \citet{BhMa}.\footnote{An improved version of the proof of this result of \citet{BhMa} can be found in the arXiv version of this paper \citep{FSTW}.} It shows convergence of the process $X_t$ to a unique stationary distribution under a monotone mixing or splitting condition. Recall that a distribution, say $\pi$, is \emph{stationary\/} for a Markov chain $X_t$, $t=0,1,\ldots$ if taking the initial value $X_0$ with distribution $\pi$
implies that all $X_t$, $t\ge 1$ also have distribution $\pi$. Results of this type were originally obtained in \citet{Dubins} (under an
additional assumption of continuity of the mapping $f$).
\end{jet}
\begin{arxiv}
We also denote by $F_{t}^{(x)}$ the distribution function of the random
variable $X_{t}$ if $X_{0}=x$ (and more generally denote by $F_{t}^{(\mu_0)}$
the distribution function of $X_{t}$ if $X_{0}$ has distribution $\mu_0$). Our first Theorem reproduces a result of \citet{BhMa}.\footnote{An improved version of the proof of this result of \citet{BhMa} can be found in the Appendix.} It shows convergence of the process $X_t$ to a unique stationary distribution under a monotone mixing or splitting condition. Recall that a distribution, say $\pi$, is \emph{stationary\/} for a Markov chain $X_t$, $t=0,1,\ldots$ if taking the initial value $X_0$ with distribution $\pi$
implies that all $X_t$, $t\ge 1$ also have distribution $\pi$. Results of this type were originally obtained in \citet{Dubins} (under an
additional assumption of continuity of the mapping $f$).
\end{arxiv}
\begin{theorem}
\label{BhMa2} Assume that time-homogeneous Markov chain
$X_{t}$ is represented by the stochastic recursion \eqref{eq:generalsetting}
with i.i.d.\ driving sequence $\{Z_{t}\}$, where function $f:[a,b]\times
\mathcal{Z}\rightarrow\lbrack a,b]$ is monotone increasing in the first
argument.\newline Assume there exists a number $c\in\lbrack a,b]$ and integer
$N\geq1$ such that
\[
\varepsilon_{1}\colonequals{\mathbf P}^{(b)}(X_{N}\leq c)>0
\]
and
\[
\varepsilon_{2}\colonequals{\mathbf P}^{(a)}(X_{N}\geq c)>0.
\]
Then, there exists a distribution $\pi$ on $[a,b]$ such that,
for any initial distribution $\mu_0$,
\begin{equation}
\sup_{x}d(F_{t}^{(\mu_0)},\pi)\rightarrow0,\quad t\rightarrow\infty\label{conv1}%
\end{equation}
exponentially fast.
\newline Furthermore, $\pi$ is the unique stationary
distribution for the Markov chain $X_{t}$.
\end{theorem}


\begin{rem}
{Theorem~1 is easily generalized to a case where the set $\mathcal{S}$ has a
partial order, $\leq$, such that there exists a least element $s_{0}%
\in\mathcal{S}$ and greatest element $s_{1}\in\mathcal{S}$ and $f$ is monotone
increasing in the first argument (with respect to the partial order $\leq
$).\footnote{{In this case, the mixing condition requires that there exists an
$\varepsilon>0$, an integer $N\geq1$ and sets $\mathcal{C}_{u}\subset
\mathcal{S}$ and $\mathcal{C}_{l}\subset\mathcal{S}$ such that for every
element $s\in\mathcal{S}$, there either exists an element $c\in\mathcal{C}%
_{u}$ such that $s\geq c$, or there exists an element $c\in\mathcal{C}_{l}$
such that $s\leq c$; and for every $c\in\mathcal{C}_{u}$, ${\mathbf{P}%
}^{(s_{1})}(X_{N}\leq c)>\varepsilon$, and for every $c\in\mathcal{C}_{l}$,
${\mathbf{P}}^{(s_{0})}(X_{N}\geq c)>\varepsilon$.}}}
\end{rem}


\subsection{Regenerative driving process}\label{subsec:regen}

We now turn our attention to the general regenerative setting
\eqref{eq:generalsetting}, but continue to assume \eqref{eq:zero-one} to hold, {that is, that the state space $\cal{X}$ is a closed interval.\footnote{This is less restrictive than it may seem because even when the state space is unbounded, it may be possible to show that all states outside of the closed interval are transient and the state must end up in the closed interval.}}

{We generalize Theorem~\ref{BhMa2} to this setting. The way this is done is first to apply Theorem~\ref{BhMa2} to the regeneration times using the i.i.d.\ nature of the cycles between the regeneration times. This implies convergence to a distribution $\pi$ at the regeneration times. Next, convergence for all dates can be established using the fact that the probabilistic nature of all cycles after the first is the same and that each cycle will in the limit start from the same distribution $\pi$. This stationary distribution for $X_t$, say $\mu$ may, in general, differ from $\pi$ and we give a simple example below (Example~\ref{queueexample}) where they do differ.}

{To proceed with the first step we} introduce an auxiliary process $\widetilde{X}_{t}^{(\alpha)}$ that starts
from $\widetilde{X}_{0}^{(\alpha)}=\alpha$ at time $0$, and follows the
recursion
\[
\widetilde{X}_{t+1}^{(\alpha)}=f\left(  \widetilde{X}_{t}^{(\alpha)}%
,Z_{T_{0}+t}\right)  \quad\text{for all}\quad t\geq0.
\]
The auxiliary process $\widetilde{X}_{t}^{(\alpha)}$ coincides in distribution
with the process $X$ started at time $T_{0}$ (i.e., at the start of
the first full cycle) from the state $\alpha$, and assumptions \eqref{eq:top}
and \eqref{eq:bottom} below ensure the mixing (similar to that guaranteed by
conditions of Theorem \ref{BhMa2}) over a typical cycle (from $T_{0}$ to
$T_{1}$) of the regenerative process $Z$. More generally, we consider an
auxiliary process $\widetilde{X}_{t}^{(F)}$ that follows the recursion
\[
\widetilde{X}_{t+1}^{(F)}=f\left(  \widetilde{X}_{t}^{(F)},Z_{T_{0}+t}\right)
\quad\text{for all}\quad t\geq0
\]
and that starts from a random variable $\widetilde{X}_{0}^{(F)}$ that has
distribution $F$ (and which does not depend on random variables $\{Z_{T_{0}%
+t},t\geq0\}$. Denote by $f^{(k)}$ the $k$-th iteration of
function $f$, so $f^{(1)}=f$ and, for, $k\geq1$,
\[
f^{(k+1)}(x,u_{1},\ldots,u_{k+1})=f\left(  f^{(k)}(x,u_{1},\ldots
,u_{k}),u_{k+1}\right),
\]
and let $f^{(0)}$ be the identity function.

\begin{theorem}
\label{thm:general} Assume that recursive sequence $\{X_{t}\}$ is defined by
\eqref{eq:generalsetting} where the function $f$ is monotone increasing in the
first argument and the sequence $\{Z_{t}\}$ is regenerative with regenerative
times $\{T_{n}\}$ that satisfy conditions \eqref{eq:finitemean}-\eqref{eq:aperiodicity}.

Assume that there exists a $c\in[a,b]$ such that the following conditions hold:
\begin{equation}
\varepsilon_{1}\colonequals{\mathbf P}\left(  \widetilde{X}_{T_{1}-T_{0}%
}^{(b)}\leq c\right)  >0, \label{eq:top}%
\end{equation}
and
\begin{equation}
\varepsilon_{2}\colonequals{\mathbf P}\left(  \widetilde{X}_{T_{1}-T_{0}%
}^{(a)}\geq c\right)  >0. \label{eq:bottom}%
\end{equation}

Then there exists a distribution ${\pi}$ on $[a,b]$ such that
\begin{equation}
\rho_{t}\colonequals\sup_{x}d(G_{n}^{(x)},\pi)=\sup_{x}\sup_{r}|G_{n}%
^{(x)}(r)-{\pi}(-\infty,r]|\rightarrow0,\quad n\rightarrow\infty
\label{eq:conv}%
\end{equation}
exponentially fast. Here $G_{n}^{(x)}$ is the distribution of $X_{T_{n}}$ if
$X_{T_{0}}=x$.

Furthermore,
the distributions of $X_{t}$ converge in the uniform metric to distribution
\[
\mu(\cdot)=\frac{1}{{\mathbf{E}}(\tau_{1})}\sum_{l=0}^{\infty}{\mathbf{P}%
}\left(  \tau_{1}>l,f^{(l)}\left(  \widetilde{X}_{0}^{(\pi)},Z_{T_{0}}%
,\ldots,Z_{T_{0}+l-1}\right)  \in\cdot\right)
\]
for any initial value $X_0$.

The following also holds for the joint distributions of $(X_t,Z_t)$:
\begin{multline*}
\sup_r \sup_{A \in \mathcal{B_Z}}\biggl|\mathbf{P}\left(X_t \le r, Z_t \in A\right)
\\ - \frac{1}{{\mathbf{E}}(\tau_{1})}\sum_{l=0}^{\infty}{\mathbf{P}%
}\left(  \tau_{1}>l,f^{(l)}\left(  \widetilde{X}_{0}^{(\pi)},Z_{T_{0}}%
,\ldots,Z_{T_{0}+l-1}\right)  \le r, Z_{T_{0}+l-1} \in A \right)\biggr| \to 0
\end{multline*}
as $t \to \infty$, for any initial value $X_{0}$.
\end{theorem}

\begin{rem}
Note that, as in the Markovian case of Theorem~1, we do not require the
function $f$ to be continuous in the first argument.
\end{rem}

\begin{rem}
In general, we require only the first moment of $\tau_{1}$ to be finite, so
convergence in the regeneration theorem may be arbitrarily
slow, and
the same holds for convergence of the distribution $F_{t}$ of random variable
$X_{t}$ to $\mu$. However, if $\tau_{1}$ has finite $(1+r)$-th moment, then
$d(F_{t},\mu)$ decays no slower than $t^{-r}$; and if $\tau_{1}$ has finite
exponential moment, then the convergence is exponentially fast.
\end{rem}

\begin{rem}
The mixing conditions \textbf{\eqref{eq:top}-\eqref{eq:bottom}}
are required to apply over a single regenerative cycle. However this
is not restrictive as a new cycle can be defined for example to consist of
appropriate multiple occurrences of an original cycle.
\end{rem}

The following simple example illustrates an application of the theorem and computation of the limiting distribution. It also shows that the distributions $\pi$ and $\mu$ in Theorem 2
may, in general, be different.

\begin{example}\label{queueexample}
Consider a simple example, with only two states of environment
$\mathcal{V}=\{1,2\}$ and with four-state space $\mathcal{X}=\{0,1,2,3\}$
(i.e., $[a,b]$ $=$ $[0,3]$). Assume sequence $\{V_{t}\}$ to be regenerative,
with the typical cycle taking two values, $(2,1)$ and $(2,2,1)$, with equal
probabilities $1/2$, so the cycle length $\tau_{1}$ is either 2 or 3, with
mean ${\mathbf{E}}\tau_{1}=5/2.$ Let $\{\xi_t^1\}$ and $\{\xi_t^2\}$
be two mutually independent i.i.d.\ sequences with the following distributions:
 ${\mathbf{P}}(\xi_{t}^{1}=k)=1/4$ for $k=0,1,2,3$ and ${\mathbf{P}}%
(\xi_{t}^{2}=-1)={\mathbf{P}}(\xi_{t}^{2}=-2)=1/2$.
Now define the driving sequence $Z_t$ as $Z_t=\xi_t^{V_t}$. The stochastic recursion is given by
\[
X_{t+1}=\min(3,\max(0,X_{t}+Z_t)), \quad t=0,1,\ldots .
\]
It may be easily checked
that the SRS satisfies all the conditions of the previous theorem.

Introduce the embedded Markov chain $Y_{n}=X_{T_{n}}$, as in the proof of the
previous theorem. It is irreducible with transition probability matrix
$P=\{p_{i,j},0\leq i,j\leq3\}$ given by
\[
\setlength{\delimitershortfall}{0pt}P=%
\begin{pmatrix}
\tfrac{1}{4} & \tfrac{1}{4} & \tfrac{1}{4} & \tfrac{1}{4}\\[1ex]%
\tfrac{1}{4} & \tfrac{1}{4} & \tfrac{1}{4} & \tfrac{1}{4}\\[1ex]%
\tfrac{3}{16} & \tfrac{1}{4} & \tfrac{1}{4} & \tfrac{5}{16}\\[1ex]%
\tfrac{3}{32} & \tfrac{3}{16} & \tfrac{1}{4} & \tfrac{15}{32}%
\end{pmatrix}
.
\]
For example, here
\begin{align*}
p_{3,1}  &  ={\mathbf{P}}(\tau_{1}=2,\xi_{1}^{2}=-2,\xi_{2}^{1}=0)+{\mathbf{P}%
}(\tau_{1}=3,\xi_{1}^{2}=\xi_{2}^{2}=-1,\xi_{3}^{1}=0)\\
&  +{\mathbf{P}}(\tau_{1}=3,\xi_{1}^{2}+\xi_{2}^{2}=-3,\xi_{3}^{1}%
=1)+{\mathbf{P}}(\tau_{1}=3,\xi_{1}^{2}=\xi_{2}^{2}=-2,\xi_{3}^{1}=1)\\
&  =\frac{1}{16}+\frac{1}{32}+\frac{1}{16}+\frac{1}{32}=\frac{3}{16}.
\end{align*}
Then the distribution of $Y_{n}$ converges to $\pi=(\pi_{0},\pi_{1},\pi
_{2},\pi_{3})$ which may be found by solving $\pi P=\pi$ with $\sum\pi_{i}=1$.
So we get $\pi=(29/160,183/800,1/4,17/50)$. Furthermore, the limiting
distribution for $X_{t}$ is given by
\begin{align*}
\mu_{k}  &  =\frac{1}{{\mathbf{E}}\tau_{1}}({\mathbf{P}}(Y^{(0)}%
=k)+{\mathbf{P}}(\max(0,Y^{(0)}+\xi_{0}^{2})=k)\\
&  +{\mathbf{P}}(\max(0,Y^{(0)}+\xi_{0}^{2}+\xi_{1}^{2})=k,\tau_{1}=3)),
\end{align*}
for $k=0,1,2,3$, where $Y^{(0)}\sim\pi$. In particular, $\mu_{3}=2\pi_{3}/5$,
$\mu_{2}=\frac{2}{5}(\pi_{2}+\pi_{3}/2)$, and $\mu_{1}=\frac{2}{5}(\pi
_{1}+(\pi_{2}+\pi_{3})/2+\pi_{3}/8)=\frac{2}{5}(\pi_{1}+\pi_{2}/2+5\pi
_{3}/8).$
\end{example}

\subsection{The case where the governing sequence is Markov}\label{subsec:markov}

In the particular case where $\{Z_{t}\}$ is a Markov chain on a countable
state space, Theorem~\ref{thm:general} leads to the following corollary, which
is important for two of the examples considered in the next section.

\begin{cor}
\label{Cor1} Assume again that the recursive sequence ~$\{X_{t}\}$ is defined
by \eqref{eq:generalsetting}, and that the function $f$ is
monotone increasing in the first argument. Assume in addition that~$\{Z_{t}\}$
is an aperiodic Markov chain on a countable state space with a positive
recurrent atom at point~$z_{0}$.
Assume also
that there exists a number~$a\leq c\leq b$, positive integers~$N_{1}$ and
$N_{2}$ and sequences $z_{1,1},\ldots,z_{N_{1},1}$ and $z_{1,2},\ldots
,z_{N_{2},2}$ such that $z_{N_{1},1}=z_{N_{2},2}=z_{0}$ and, for $i=1,2$, the
following hold:
\[
p_{i}\colonequals{\mathbf P}(Z_{j}=z_{j,i},\ \mbox{for}\ \ j=1,\ldots
,N_{i}\ |\ Z_{0}=z_{0})>0
\]
and that
\[
\delta_{1}\colonequals\mathbf{P}(\widetilde{X}_{N_{1}}^{(b)}\leq
c\ |\ Z_{0}=z_{0},\ Z_{j}=z_{j,1},\ j=1,\ldots,N_{1})>0
\]
and
\[
\delta_{2}\colonequals\mathbf{P}(\widetilde{X}_{N_{2}}^{(a)}\geq
c\ |\ Z_{0}=z_{0},\ Z_{j}=z_{j,2},\ j=1,\ldots,N_{2})>0.
\]
Then the distribution of~$X_{t}$ converges in the uniform metric to a unique
stationary distribution.
\newline
{In addition, There exists a stationary sequence $(X^{t},Z^{t})$ such that, as
$t\rightarrow\infty$,
\[
\sup_{a\leq x\leq b}\sup_{B\in\mathcal{B_{Z}}}|{\mathbf{P}}(X_{t}\leq
x,Z_{t}\in B)-{\mathbf{P}}(X^{t}\leq x,Z^{t}\in B)|\rightarrow0.
\]
}
\end{cor}
\begin{rem}
\label{genMark} For simplicity, we have assumed that the Markov chain in
Corollary~\ref{Cor1} is defined on a countable state space with a positive
recurrent atom. However, Corollary \ref{Cor1} can be extended to the case of a
driving Markov chain on a general state space provided a \textquotedblleft
Harris-type\textquotedblright\ condition is satisfied. Here we outline the
conditions required. Consider again a recursive sequence $\{X_{t}\}$ with the
function $f$ monotone increasing in the first argument. Assume that there
exists a measurable set $A$ in the state space $(\mathcal{Z},\mathcal{B_{Z}})$
that is \textit{positive recurrent}:
\[
T_{1}(z_{0})=\min\{t>0\ :\ Z_{t}^{(z_{0})}\in A\}<\infty\quad\mbox{a.s.,
for any}\quad z_{0}\in\mathcal{Z}%
\]
and
\[
\sup_{z_{0}\in A}{\mathbf{E}}T_{1}(z_{0})<\infty.
\]
Here $Z_{t}^{(z_0)}$ is a 
Markov chain with initial value $Z_{0}^{(z_{0})}=z_{0}$. Furthermore, assume
that there exist positive integers $N_{1}$ and $N_{2}$, a positive number
$p\leq1$ and a probability measure $\varphi$ on $A$ such that, for $i=1,2$ and
all $z_{0}\in A$,
\[
{\mathbf{P}}(Z_{N_{i}}^{(z_{0})}\in\cdot)\geq p\varphi(\cdot)
\]
and that there exists a number $a\leq c\leq b$ and positive numbers
$\delta_{1}$ and $\delta_{2}$ such that
\[
{\mathbf{P}}(\widetilde{X}_{N_{1}}^{(b)}\leq c\ |\ Z_{0}=z_{0},Z_{N_{1}}%
=z_{1})\geq\delta_{1}%
\]
and
\[
{\mathbf{P}}(\widetilde{X}_{N_{2}}^{(a)}\geq c\ |\ Z_{0}=z_{0},Z_{N_{2}}%
=z_{2})\geq\delta_{2},
\]
for $\varphi$-almost surely all $z_{0},z_{1},z_{2}\in A$. With all these
conditions and aperiodicity of the Markov chain, it can be shown that the
distribution of $X_{t}$ converges in the uniform metric to the unique
stationary distribution.
\end{rem}


\subsection{Discussion}\label{subsec:discuss}

\begin{jet}
In this section we discuss our assumption that the state space is a closed interval of the real line and the relation of our results to some of the existing literature.

We first note that our results can be extended to a state space that is partially ordered. For example, if the state space is $[a,b]\times [a,b]$ (with the natural partial ordering: $(x_1,x_2)\le (y_1,y_2)$ if and only if
$x_1\le y_1$ and $x_2\le y_2$), then our results apply with only minor and natural modifications.

The extension to the case where the state space is not compact is however, likely to be more complicated. In particular, in this case, the $\delta$ and $N$ in our Corollary~1, may depend on initial conditions. Both \citet{KaSta} and \citet{Sz} consider the Markov chain model of \citet{BhMa} (with i.i.d.\ driving sequence) and establish convergence and uniqueness for monotone economies. \citet{KaSta} introduce a \emph{strong reversing condition\/} that requires that if there are two mutually independent trajectories $X_{t}^{(y_{0})}$ and $X_{t}^{(x_{0})}$ for the pair of initial conditions $x_{0}<y_{0}$, then, there is an $N\geq1$ and $\delta>0$ such that ${\mathbf{P}}(X_{N}^{(y_{0})}\leq X_{N}^{(x_{0})})\geq\delta$. This is, of course, equivalent to the monotone mixing condition of Theorem~\ref{BhMa2} when the state space is $[a,b]$ but can be applied when the state space is non-compact as well.

\citet{Sz} considers a model with an ordered state space that has no maximal and/or minimal element. The author suggested a reasonable
``replacement'', say, for a maximal element (if one
does not exist) by a random \textquotedblleft top\textquotedblright\ point.
In our notation, this generalization may be viewed as follows.
Assume, say, the state space for the Markov chain is the positive
half-line $[0,\infty)$ where there is no
maximal element, and suppose that a Markov chain $X_{t}$ is defined by a
stochastic recursion $X_{t+1}=f(X_{t},\xi_{t})$ with i.i.d.\ $\{\xi_{t}\}$.
Assume that there exists a random measure $\mu$ on $[0,\infty)$ such that if
$X_{0}\sim\mu$ and if $X_{0}$ does not depend on $\xi_{0}$, then
$X_{1}=f(X_{0},\xi_{0})$ is \emph{stochastically smaller\/} than $X_{0}$ (that is, ${\mathbf{P}}(X_{1}\leq x)\geq{\mathbf{P}}(X_{0}\leq x)$, for all $x$). In this case, the distribution $\mu$ may play a role of a new random \textquotedblleft
top\textquotedblright\ point if, for example, the distribution of $\mu$ has an
unbounded support.
\end{jet}
\begin{arxiv}
In this section we discuss our assumption that the state space is a closed interval of the real line and the relation of our results to some of the existing literature.

We first note that our results can be extended to a state space that is partially ordered. For example, if the state space is $[a,b]\times [a,b]$ (with the natural partial ordering: $(x_1,x_2)\le (y_1,y_2)$ if and only if
$x_1\le y_1$ and $x_2\le y_2$), then our results apply with only minor and natural modifications.

The extension to the case where the state space is not compact is however, likely to be more complicated. In particular, in this case, the $\delta$ and $N$ in our Corollary~1, may depend on initial conditions. Both \citet{KaSta} and \citet{Sz} consider the Markov chain model of \citet{BhMa} (with i.i.d.\ driving sequence) and establish convergence and uniqueness for monotone economies. \citet{KaSta} introduce a \emph{strong reversing condition\/} that requires that if there are two mutually independent trajectories $X_{t}^{(y_{0})}$ and $X_{t}^{(x_{0})}$ for the pair of initial conditions $x_{0}<y_{0}$, then, there is an $N\geq1$ and $\delta>0$ such that ${\mathbf{P}}(X_{N}^{(y_{0})}\leq X_{N}^{(x_{0})})\geq\delta$. This is, of course, equivalent to the monotone mixing condition of Theorem~\ref{BhMa2} when the state space is $[a,b]$ because the mixing condition may be written as there is an
$N\geq1$ and $\delta>0$ such that ${\mathbf{P}}(X_{N}^{(b)}\leq X_{N}^{(a)})\geq\delta>0$ and by monotonicity the same holds for any other pair of initial conditions $a\leq x_{0}<y_{0}\leq b$, with the same $\delta$ and $N$.

\citet{Sz} considers a model with an ordered state space that has no maximal and/or minimal element. The author suggested a reasonable
``replacement'', say, for a top point (if one
does not exist) by a random \textquotedblleft top\textquotedblright\ point.
In our notation, this generalization may be viewed as follows.
Assume, say, the state space for the Markov chain is the positive
half-line $[0,\infty)$ where there is no
maximal element, and suppose that a Markov chain $X_{t}$ is defined by a
stochastic recursion $X_{t+1}=f(X_{t},\xi_{t})$ with i.i.d.\ $\{\xi_{t}\}$.
Assume that there exists a random measure $\mu$ on $[0,\infty)$ such that if
$X_{0}\sim\mu$ and if $X_{0}$ does not depend on $\xi_{0}$, then
$X_{1}=f(X_{0},\xi_{0})$ is \emph{stochastically smaller\/} than $X_{0}$ (that is, ${\mathbf{P}}(X_{1}\leq x)\geq{\mathbf{P}}(X_{0}\leq x)$, for all $x$). In this case, the distribution $\mu$ may play a role of a new random \textquotedblleft
top\textquotedblright\ point if, for example, the distribution of $\mu$ has an
unbounded support. For instance, if there exists another function, say $h$
such that $f(x,y)\leq h(x,y)$ for all $x,y$ and that a Markov chain
$Y_{t+1}=h(Y_{t},\xi_{t})$ admits a unique stationary distribution, say $\mu$.
If $\mu$ can be easily determined, then it may play the role of a random
\textquotedblleft top\textquotedblright\ point.

Here is a simple example. Assume that $X_{t}$ is a discrete-time
birth-and-death-process with immigration at $0$, i.e. a non-negative
integer-valued Markov chain, which is homogeneous in time and with transition
probabilities ${\mathbf{P}}(X_{1}=1\ |\ X_{0}=0)=1-{\mathbf{P}}(X_{1}%
=0\ |\ X_{0}=0)=p_{0}>0$ and, for $k=1,2,\ldots,$ let ${\mathbf{P}}%
(X_{1}=k+1\ |\ X_{0}=k)=1-{\mathbf{P}}(X_{1}=k-1\ |\ X_{0}=k)=p_{k}$.
Furthermore assume that the $p_{k}$ are non-decreasing in $k$ (this makes the
Markov chain monotone), that all are smaller than $1/2$ and, moreover, that
$\lim_{k\rightarrow\infty}p_{k}=p<1/2$. Consider a Markov chain $Y_{t}$ with
transition probabilities ${\mathbf{P}}(Y_{1}=k+1\ |\ Y_{0}=k))=p=1-{\mathbf{P}%
}(Y_{1}=\max(0,k-1))\ |\ Y_{0}=k)$. Then this Markov chain has a unique
stationary distribution $\mu$ (which is clearly geometric), and it gives a
random \textquotedblleft top\textquotedblright\ point.
\end{arxiv}

We believe that this approach may be extended further to stochastic recursive sequences with regenerative drivers, using the ideas from
 \citet{BoFo}, where a similar concept of a stationary majorant was developed and
studied, using the construction of a stationary top sequence $\{X_{t}\}$. See further \citep{FTw} and \citep{TwCor} (and the references therein) where
similar ideas have been developed in the context of
``perfect simulation from the past'', with the introduction
of an artificial random \textquotedblleft top\textquotedblright\ point.

{Finally, we mention the paper of \citet{Acemoglu-Jensen15} that considers comparative static properties in similar setting with large numbers of agents (similar to the application we consider in Section~\ref{subsec:huggett}) and a Markov driving process. Their focus, however, is on developing results for any equilibrium distribution and not in establishing uniqueness.}

\section{Applications \label{sec:eg}}

In this section we present three workhorse models. In the first we allow for
general driving process as in Example \ref{markovexample}. In the second and
third we assume that the driving process is a Markov chain and apply
Corollary~1. We are thus able to extend known stability results in these models.

\subsection{Bewley-Imrohoroglu-Huggett-Aiyagari Precautionary Savings
Model\label{subsec:huggett}}

The basic income fluctuation model in which many risk-averse agents
self-insure against idiosyncratic income shocks through borrowing and saving
using a risk-free asset, is designated by \citet{Heathcoteetal}
\textquotedblleft the standard incomplete markets model\textquotedblright\ and
is the workhorse model in quantitative macroeconomics. At its heart is the
stochastic savings models of \citet{Huggett93} with an exogenous borrowing
constraint, or close variants of this model.\footnote{\citet{Bewley86} and
\citet{Aiyagari94} vary the context but the individual savings problem is
similar. They each derive existence and convergence results under slightly
different assumptions. \citet{Bewley86} assumes that the endowment shocks are
stationary Markov, \citet{Huggett93} assumes positive serial correlation and
two states, and \citet{Aiyagari94} assumes that endowment shocks are i.i.d.
\citet{Imrohoroglu92} uses numerical methods with a two state persistent
income process as in \citet{Huggett93}.} As \citet{Heathcoteetal} observe,
there are \textquotedblleft few general results that apply to this class of
problems.\textquotedblright\ Existing published work in the standard model
requires either that income fluctuations are i.i.d., or that an individual's
income process satisfies \textquotedblleft persistence\textquotedblright: a
higher income today implies that income tomorrow is higher in the stochastic
dominance sense. That is, that the income process is monotone.
\citet{Huggett93} has a two-state process for income and uses the
\citet{HopenhaynPrescott92} approach to prove convergence of the asset
distribution to a unique invariant distribution.\footnote{\citet{Miao02}
extends Huggett's model from two states to many states.} In the case of two
income states, the assumption of persistence in the income process is probably
innocuous. However, it may be restrictive in other cases. Obvious examples of
non-monotone processes would include termination pay where a worker receives a
large one-off redundancy payment followed by a long spell of unemployment, or
health shocks where an insurance payout is received but future employment
prospects are diminished.\footnote{As another example, consider the case where
there are a group of entrepreneurs who have very high income. It may be
possible that these entrepreneurs have a higher chance to fall to very low
income levels than those on medium income levels. This is the situation
described by \citet{Kaymak2015} who use information from observed
distributions of income and wealth to construct a transition matrix for
income. The transition matrix they use does not satisfy monotonicity.} In what
follows we consider Huggett's model with a potentially uncountable number of
states and dispense with the assumption that the income process is monotone
(we maintain all his other assumptions). Applying our methodology, we show
convergence in the uniform metric to a unique invariant
distribution.\footnote{In the subsequent analysis, we follow Huggett and
assume that the gross interest rate, $R$, is fixed. This is an ingredient into finding the equilibrium rate at which assets are in zero net demand.}$^{,}$\footnote{In independent work, and in a
more general context, \citet[][Proposition~5]{Acikgoz16} shows using different
methods that if the income process is a finite (irreducible aperiodic) Markov
chain, and there exists a \textquotedblleft worst\textquotedblright\ positive
probability sequence of incomes which is dominated at each date by any other
positive probability sequence (e.g., if the lowest income state recurs with
positive probability), then there exists a unique stationary distribution.
\citet{Zhu13}, who considers an income fluctuation model with endogenous labor
supply, uses a related argument under the assumption that the finite Markov
chain has strictly positive transition probabilities; results also hold for
the case where $\beta R =1$.}

Agents maximize expected discounted utility
\[
\mathbf{E}\left[  \sum_{t=0}^{\infty}\beta^{t}u(c_{t})\right]  ,
\]
where $c_{t}\in\mathbb{R}_{+}$ is consumption at time $t$, $t=0,1,\ldots,$
$u(c)=c^{1-\gamma}/(1-\gamma)$, $\gamma>1$, subject to a budget constraint at
each date
\begin{equation}
c+R^{-1}x^{+}\leq x+e, \label{eqn:budget}%
\end{equation}
a borrowing constraint $x^{+}\geq\underaccent{\bar}{x}$, where $e$ is the
current endowment, $x$ is current assets, $x^{+}$ is assets next period, $c$
is consumption, $\beta\in(0,1)$ is the discount factor and $R^{-1}>\beta$ is
the price of next-period assets. The individual's endowment at time $t$,
$e_{t}$, is drawn from a set $E=[\underaccent{\bar}{e},\bar{e}]$, where
$\infty>\bar{e}>\underaccent{\bar}{e}>0$; $e_{t}$ is governed by an aperiodic
\emph{positive recurrent Markov chain with an atom, }with aperiodic
regenerative times, as in Example \ref{markovexample}, where we denote by
$Q:E\times$ $\mathcal{E\rightarrow\lbrack}1,0\mathcal{]}$ the (stationary)
transition function, with $\mathcal{E}$ the Borel sets of $E$, and we assume
that the Feller property is satisfied
\citep[see, e.g.,][ch.8]{Stokey-Lucas89}. The borrowing constraint satisfies
$\underaccent{\bar}{x}<0$ and
$\underaccent{\bar}{x}+\underaccent{\bar}{e}-\underaccent{\bar}{x}R^{-1}>0$.
The initial values $e_{0}\in E$ and $x_{0}\geq\underaccent{\bar}{x}$ are given.

The individual's decision problem can be represented by the functional
equation:
\begin{equation}
v(x,e)=\max_{(c,x^{+})\in\Gamma(x,e)}u(c)+\beta\mathbf{E}\left[  v_(x^{+},e^+)\mid e\right]  \label{eqn:problem}%
\end{equation}
where $\mathbf{E}$ is expectation over $e^{+}$ given $e$, $v(x,e)$ are the value
functions, and
\[
\Gamma(x,e)=\left\{  (c,x^{+})\mid c+R^{-1}x^{+}\leq x+e,x^{+}\geq
\underaccent{\bar}{x},c\geq0\right\}
\]
is the constraint set. The resulting policy functions are denoted $c=c(x,e)$
and $x^{+}=f(x,e)$ (i.e., an optimal policy must satisfy these a.s.).
\citet[][Theorem 1]{Huggett93} proves that there is a unique, bounded and
continuous solution to (\ref{eqn:problem}) and each $v(x,e)$ is increasing,
strictly concave and continuously differentiable in $x$, while $f$ is
continuous and nondecreasing in $x$, and (strictly) increasing whenever
$f(x,e)>\underaccent{\bar}{x}$. These results extend to our context with a
continuous state space; see \citet{Miao02}.

Huggett assumes monotonicity of the endowment process: with two endowment
states, $E=\left\{  \underaccent{\bar}{e},\bar{e}\right\}$, this means
$p(\underaccent{\bar}{e},\underaccent{\bar}{e})\geq p(\bar{e}%
,\underaccent{\bar}{e})$ where $p(e,e^{\prime})$ denotes the transition
probability. He shows that for a given $R$, there exists a unique stationary
probability measure for $x=(x,e)$ and that there is weak convergence to this
distribution for any initial distribution on $x$ (see \citet[][Theorem 2]{Huggett93}).

We can extend this result to our more general context (non-discrete state
space, no monotonicity assumption) using the following {(the proof can be
found in the Appendix)}.\footnote{A similar result is established in
\citet{Huggett93}, and in \citet{Miao02} for the many state case, but using
monotonicity.}

\begin{lem}
There exists $\widehat{x}\geq\underaccent{\bar}{x}$ such that for all
$x>\widehat{x}$, all $e\in E$, $f(x,e)<x$.
\end{lem}

Given this, we can restrict attention to $[\underaccent{\bar}{x},\hat{x}]$ and
convergence follows from the following argument.\footnote{The details of the
argument are presented in the Appendix.} Starting from $\hat{x}$, there must
be some positive probability of hitting the credit constraint: given
$R^{-1}>\beta$ the only reason for holding assets above
$\underaccent{\bar}{x}$ is the precautionary one, and never hitting
$\underaccent{\bar}{x}$ would imply that assets are excessive, so
$\underaccent{\bar}{x}$ must be hit at some time $T$ with positive
probability. Because $f\left(  x,e\right)  $ is nondecreasing in $x$, starting
at $\underaccent{\bar}{x}$ instead of at $\hat{x}$ but with the same sequence
of endowment shocks, implies that assets at $T$ are also at
$\underaccent{\bar}{x}$. This implies that the mixing condition of Theorem 2
is satisfied at the end of a regenerative cycle suitably defined (by the next
occurrence of the atom after $T$). Thus there exists a unique distribution
$\pi$ on $[\underaccent{\bar}{x},\hat{x}]$ such that the distributions of
$x_{t}$ converge to $\pi$ in the uniform metric for any initial value
$x_{0}\in\lbrack\underaccent{\bar}{x},\hat{x}]$.

\subsection{One-Sector Stochastic Optimal Growth Model}\label{subsec:growth}

The Brock-Mirman \citep{BrockMirman72} one-sector stochastic optimal growth
model has been extended to the case of correlated production shocks by
\citet{DonaldsonMehra} and \citet[][pp.\ 1402--03]{HopenhaynPrescott92}. With
correlated productivity shocks, it is possible to prove uniqueness and
convergence results using the methods of \citet{HopenhaynPrescott92} or
\citet[][Chapter~12]{Stokey-Lucas89} provided the policy function for the
investment is itself monotonic in the productivity shock. Although the
assumption of correlated shocks is very reasonable in this context,
establishing that the policy function is monotone in the productivity shock
is, as pointed out by \citet[][pp.\ 1403]{HopenhaynPrescott92}, difficult
without imposing very restrictive assumptions. The reason is simple. A good
productivity shock today increases current output, which may allow increased
investment. However, because shocks are positively correlated, output will
also be higher on average tomorrow and hence consumption can be too.
Therefore, it may be desirable to increase current consumption by more than
the increase in current output, cutting back on current
investment.\footnote{The sufficient condition given in
\citet{HopenhaynPrescott92} for monotonicity of the policy function in the
productivity shock is
\begin{gather*}
\frac{f_{kz}}{f_{k}\cdot f_{z}}\geq-\frac{u^{\prime\prime}}{u^{\prime}},
\end{gather*}
where $f$ is the production function, depending on capital $k$ and
productivity shock $z$, and $u$ is the utility function. Since the arguments
of the utility function and production function depend on the policy function
themselves, this condition is difficult to check a priori, except in special
cases. One such special case is where the capital and productivity shock are
perfect complements in production, in which case the left-hand-side of the
above inequality becomes infinitely large.} Since our results do not require
monotonicity of the policy function in the driving process, we can establish
convergence to a unique invariant distribution without requiring any extra
restrictive conditions on preferences and productivity beyond those normally
assumed in the stochastic growth model. In addition, of course, we do not
require the productivity shocks to be positively correlated.

We consider a version of the Brock-Mirman one sector stochastic optimal growth
model with full depreciation of capital. Paths for consumption, $c_{t} $, and
capital, $k_{t}$, are chosen to
\[
\max\mathbf{E}\sum\nolimits_{t=0}^{\infty}\beta^{t}u(c_{t})
\]
subject to
\[
f(k_{t},z_{t})\geq c_{t}+k_{t+1},\quad c_{t}\geq0,
\]
where $\beta\in(0,1)$ is the discount factor, $u$ is the utility function, $f$
is the production function and $z_{t}$ is a productivity shock.\footnote{For
this section we use $f$ to denote the production function and $g$ to denote
the policy function.} The productivity shock is drawn from a finite set
$\hat{\mathcal{Z}}\colonequals\{{z}^{1},\ldots,{z}^{n}\}$, $n\geq2$, with
$z_{t}$ governed by a time-homogeneous Markov chain with  transition probabilities
$p(z,z^{+})\colonequals {\mathbf P}(z_{t+1}=z^{+}\mid z_{t}=z)>0$, for all
$z,z^{+}\in\widehat{\mathcal{Z}}$.\footnote{\citet{BrockMirman72} also assume
a finite set of states but assumed the stochastic shock process was i.i.d.} We
make some standard assumptions on preferences and technology. The utility
function $u\mathpunct{:}\mathbb{R}_{+}\rightarrow\mathbb{R}\cup\{-\infty\}$ is
continuous, strictly increasing, and strictly concave on $\mathbb{R}_{+}$ (on
$\mathbb{R}_{++}$ if $u(0)=-\infty$), with $\lim
_{c\mathrel{\nonscript\mkern-1.2mu\mkern1.2mu{\downarrow}}0}u(c)=u(0)$; it is
twice continuously differentiable for $c>0$ and $\lim
_{c\mathrel{\nonscript\mkern-1.2mu\mkern1.2mu{\downarrow}}0}u^{\prime
}(c)=\infty$. The production function $f\mathpunct{:}\mathbb{R}_{+}%
\times\widehat{\mathcal{Z}}\rightarrow\mathbb{R}_{+}$ is continuously
differentiable, strictly increasing and strictly concave in $k$ with
$\lim_{k\mathrel{\nonscript\mkern-1.2mu\mkern1.2mu{\downarrow}}0}%
f_{k}(k,z)=\infty$ for all $z\in\widehat{\mathcal{Z}}$ (where $f_{k}$ denotes
$\partial f(k,z)/\partial k$), $f(0,z)=0$ for all $z\in\widehat{\mathcal{Z}}$, and is such that there exists a $k^{\max}>0$
with $f(k,z)<k$ for all $k>k^{\max}$ and all $z\in\widehat{\mathcal{Z}}$. The
initial conditions are $k_{0}>0$ and $z_{0}\in\widehat{\mathcal{Z}}$ given.

The problem can be set up recursively. Letting $k^{+}$ denote next period's
capital stock and $z^{+}$ next period's shock, the value function satisfies
\begin{gather}
\label{eq:valuefn}v(k,z)=\max_{0\leq k^{+}\leq f(k,z)}\,u(f(k,z)-k^{+}%
)+\beta\mathbf{E}\left[  v(k^{+},z^{+})\mid z\right]
\end{gather}
where $\mathbf{E}$ is expectation over $z^{+}$ given $z$. Let $k_{t+1}%
=g(k_{t},z_{t})$ be the policy function, and
$c(k,z)\colonequals f(k,z)-g(k,z)$. The following is standard
(see, e.g., \citet[][Chapter 10]{Stokey-Lucas89}): $c(k,z)$ and $g(k,z)$ are continuous and increasing in
$k$; moreover $v(k,z)$ is increasing, strictly concave and differentiable in
$k$ for $k>0$.

For $k>0$, the solution to the maximization problem in~\eqref{eq:valuefn} is
interior.\footnote{We have $k^{+}>0$ because the marginal return to saving,
$\beta\mathbf{E}[u^{\prime}(c^{+})f_{k}(k^{+},z^{+})\mid
z]\rightarrow\infty$ as $k^{+}%
\mathrel{\nonscript\mkern-1.2mu\mkern1.2mu{\downarrow}}0$ which therefore
exceeds $u^{\prime}(c_{t})$ for all $k^{+}$ near zero. Similarly, the
condition $\lim_{c\mathrel{\nonscript\mkern-1.2mu\mkern1.2mu{\downarrow}}0}%
u^{\prime}(c)=\infty$ ensures $k^{+}<f(k,z)$.} Thus, the first-order and
envelope conditions are given by:
\begin{align}
u^{\prime}(c(k,z))  &  =\beta\mathbf{E}\left[  v_{k}(g(k,z),z^{+}%
))\mid z\right]  ,\label{FOC}\\
v_{k}(k,z)  &  =u^{\prime}(c(k,z))f_{k}(k,z). \label{envelope}%
\end{align}
Combining~\eqref{FOC} and~\eqref{envelope}, we have:
\begin{align}
v_{k}(k,z)  &  =\beta f_{k}(k,z)\mathbf{E}\left[  v_{k}(g(k,z),z^{+}%
))\mid z\right]  ,\label{vkeuler}\\
u^{\prime}(c(k,z))  &  =\beta\mathbf{E}\left[  u^{\prime
}(c(g(k,z),z^{+}))f_{k}(g(k,z),z^{+})\mid z\right]  . \label{rameuler}%
\end{align}

Define the upper and lower envelopes of the policy functions: $\bar
{g}(k)\colonequals\max_{z}g(k,z)$ and $\lowerbar{g}(k)\colonequals\min
_{z}g(k,z)$. These functions are continuous and increasing and $\bar{g}%
(k)\geq\lowerbar{g}(k)$. Define $k^{\prime\prime}\colonequals\inf
\{k>0\mid\bar{g}(k)\leq k\}$ and $k^{\prime}\colonequals\sup\{0<k\leq
k^{\prime\prime}\mid\lowerbar{g}(k)=k\}$. To establish convergence on a
positive and bounded interval, $[k^{\prime},k^{\prime\prime}]$, we first prove
the following lemma (the proof can be found in the Appendix).

\begin{lem}
\label{lem:2} (i) There is an $\epsilon>0$ such that $\lowerbar{g}(k)>k$ for
all $k\in(0,\epsilon)$; (ii)If $k^{\prime\prime}>k^{\prime}$ then for all
$k>k^{\prime}$, $\lowerbar{g}(k)<k$.
\end{lem}

The first part of the lemma adapts the arguments of \citet{Mitra-Roy12}
(see also \cite{Roy-Itzhak12}) to establish that there is growth with
probability one near zero capital. That is, the capital stock must optimally
increase if capital is close to zero and hence $k^{\prime},k^{\prime\prime}%
>0$. This result is derived from the Inada condition on the marginal product
at zero and the assumption that transition probabilities are positive. The
second part of the lemma ensures that sets above $k^{\prime\prime}$ are
transient, and allows the corollary to be applied in a straightforward manner.
Note that $k^{\prime\prime}$ exists by the continuity of $\bar{g}(k)$ and is
finite because $\bar{g}(k)\leq f(k,z)<k$ for all $k>k^{\max}$ and $z$.

Assume that a degenerate stationary equilibrium at $k>0$ does not exist (see
below for some conditions that guarantee this). Then we can establish
convergence in the uniform metric of the distributions of $k_{t}$ to a unique
non-degenerate stationary distribution $\pi$ with support in $\left[
k^{\prime},k^{\prime\prime}\right]  $ for any initial value $k_{0}>0$: First,
$k^{\prime\prime}>k^{\prime}$ since $k^{\prime\prime}=k^{\prime}$ implies
$\lowerbar{g}(k^{\prime\prime})=\bar{g}(k^{\prime\prime})$ and hence a
degenerate steady state at $k^{\prime\prime}$. Next, for any $k>0$ where
$k\not \in \lbrack k^{\prime},k^{\prime\prime}]$, it follows from the
definitions that all such $k$ are transient and there is a positive
probability sequence such that $k$ will transit to this interval. Next, with
$\bar{g}(k)>k>\lowerbar{g}(k)$ for all $k\in(k^{\prime},k^{\prime\prime})$ by
definition of $k^{\prime}$ and by part (ii) of the lemma, we can show that the
relevant mixing condition of Corollary~1 is satisfied on $[k^{\prime
},k^{\prime\prime}]$. To see this start from $(k^{\prime\prime},z_{0})$;
repeatedly applying $\lowerbar{g}$ yields the strictly decreasing sequence
$(\lowerbar{g}(g(k^{\prime\prime},z_{0})),\lowerbar{g}^{(2)}(g(k^{\prime
\prime},z_{0})),\ldots)$ where $\lowerbar{g}^{(n)}$ denotes the $n$-fold
composition of $\lowerbar{g}$. It follows that $\lim_{T\rightarrow\infty
}\lowerbar{g}^{(T)}(g(k^{\prime\prime},z_{0}))=k^{\prime}$%
.\footnote{Otherwise, if $\lim_{T\rightarrow\infty}\lowerbar{g}%
{(T)}(g(k^{\prime\prime},z_{0}))=\widetilde{k}>k^{\prime}$, then the
continuity of $\lowerbar{g}$ implies $\lowerbar{g}(\widetilde{k}%
)=\widetilde{k}$, which contradicts $\bar{g}(k)>k>\lowerbar{g}(k)$ for all
$k\in(k^{\prime},k^{\prime\prime})$.} Therefore, fixing some $\hat{k}%
\in(k^{\prime},k^{\prime\prime})$, there exists a finite sequence of
productivity shocks $(z_{t})_{t=1}^{T}$ with $z_{t}\in\argmin_{z\in
Z}\{g(\lowerbar{g}^{(t-1)}(g(k^{\prime\prime},z_{0})),z)\}$ such that the
occurrence of $(z_{t})_{t=1}^{T-1}$ implies $k_{T}\leq\hat{k}$. Moreover, the
sequence $(z_{t})_{t=1}^{T}$, with $z_{T}=z_{0}$, has positive probability
since all the transition probabilities are positive. By a symmetric argument,
using $\bar{g}(k)$ and starting from $(k^{\prime},z_{0})$, there exists a
positive probability, finite sequence of productivity shocks $(\widetilde{z}%
_{t})_{t=1}^{\widetilde{T}-1}$ whose occurrence implies $k_{\widetilde{T}}%
\geq\hat{k}$. Corollary~1 can then be applied with $c=\hat{k}$, $N_{1}%
=\widetilde{T}$, and $N_{2}=T$ to establish convergence as claimed.

Under mild conditions degenerate steady states do not exist. Here are two examples:

1. First suppose that preferences are CRRA, $u\left(  c\right)  =c^{1-\alpha
}/\left(  1-\alpha\right)  ,$ $\alpha>1$, and shocks are multiplicative with
$z\in\mathbb{R}_{++},$ $z^{1}<z^{2}<\ldots$ $z^{n}$ say, and $f(k,z)=zh\left(
k\right)  $ and write $h^{\prime}\left(  k\right)  \equiv dh/dk $.

We have $u^{\prime}f_{k}=\left(  zh\left(  k\right)  -k\right)  ^{-\alpha
}zh^{\prime}\left(  k\right)  $, and
\begin{equation}
\partial\left(  u^{\prime}f_{k}\right)  /\partial z=-\frac{h^{\prime}\left(
k\right)  ((\alpha-1)h\left(  k\right)  z+k)}{\left(  zh\left(  k\right)
-k\right)  ^{1+\alpha}}<0 \label{eq:multiplic}%
\end{equation}
by $c=zh\left(  k\right)  -k>0$. Consider (\ref{rameuler}) at a degenerate
steady state $k>0,$ where $k=g(k,z)$ all $z$:%
\begin{equation}
u^{\prime}(f(k,z)-k)=\beta\mathbf{E}\left[  u^{\prime}(f(k,z^{+}%
)-k)f_{k}(k,z^{+})\mid z\right]  . \label{rameuler1}%
\end{equation}
We have $\beta f_{k}(k,z^{n})>1$ since otherwise by $f_{k}(k,z^{n}%
)>f_{k}(k,z^{i}),$ for $i<n$, $\beta f_{k}(k,z^{i})<1$ for $i<n, $ and by
$u^{\prime\prime}<0,$ $u^{\prime}(f(k,z^{1})-k)>$ $u^{\prime}(f(k,z^{i})-k)$
for $i>1$, so we get
\[
u^{\prime}(f(k,z^{1})-k)>\beta f_{k}(k,z^{i})u^{\prime}(f(k,z^{i})-k)
\]
all $i.$ This violates (\ref{rameuler1}) for $z=z^{1}$. But then we get%
\begin{align*}
\beta f_{k}(k,z^{1})u^{\prime}(f(k,z^{1})-k)  &  >\beta f_{k}(k,z^{2}%
)u^{\prime}(f(k,z^{2})-k)>\ldots\\
&  >\beta f_{k}(k,z^{n})u^{\prime}(f(k,z^{n})-k)>u^{\prime}(f(k,z^{n})-k)
\end{align*}
where the final inequality follows by $\beta f_{k}(k,z^{n})>1$ and the rest by
(\ref{eq:multiplic}). This implies the RHS of (\ref{rameuler1}) exceeds the
LHS for $z=z^{n}$, contradicting optimality.

2. Suppose that in addition to any persistent shock to output, there is also a
transitory component to the shock (i.e., such that distribution over future
shocks is unaffected by the transitory component); specifically suppose there
exist $z^{\prime}$, $z^{\prime\prime}\in\widehat{\mathcal{Z}}$, such that
$f(k,z^{\prime})>f(k,z^{\prime\prime})$, $\forall k>0$, and $p(z^{\prime
},z)=p(z^{\prime\prime},z)$ for all $z\in\widehat{\mathcal{Z}}$.

Then the choice of next period's capital stock differs for at least two of the
possible realizations of~$z$: Taking states $z^{\prime}$ and $z^{\prime\prime
}$ as above where $f(k,z^{\prime})>f(k,z^{\prime\prime})$, it follows that
$g(k,z^{\prime})>g(k,z^{\prime\prime})$ for $k>0$, and hence there cannot be a
degenerate steady state.\footnote{Suppose otherwise, that $g(k,z^{\prime})\leq
g(k,z^{\prime\prime})$. It follows from $f(k,z^{\prime})>f(k,z^{\prime\prime
})$ that $c(k,z^{\prime})>c(k,z^{\prime\prime})$. This leads to a
contradiction of~(\ref{FOC}). The LHS of~(\ref{FOC}) is strictly lower at
$z^{\prime}$ than at $z^{\prime\prime}$ by the concavity of the utility
function. Conversely, by the concavity of $v$, $v_{k}(g(k,z^{\prime}%
),z^{+})\geq v_{k}(g(k,z^{\prime\prime}),z^{+})$ at each $z^{+}$, meaning the
the RHS of~(\ref{FOC}) is no lower at at $z^{\prime}$ than at $z^{\prime
\prime}$, by the assumption that $p(z^{\prime},z^{+})=p(z^{\prime\prime}%
,z^{+})$ all $z^{+}$.}

\subsection{Limited Commitment Risk-Sharing Model}\label{subsec:risksharing}

In this section we consider the inter-temporal risk-sharing model with limited
commitment. \citet{Kocherlakota96} \citep[see also, for
example,][]{Thomas-Worrall88,Alvarez-Jermann00,Alvarez-Jermann01,Ligon-Thomas-Worrall02}
provides a convergence result for the long-run distribution of risk-sharing
transfers when shocks to income are finite and i.i.d.
His model has two, infinitely-lived, risk averse agents with per-period,
strictly concave and differentiable utility function $u\mathpunct{:}\mathbb{R}%
_{+}\rightarrow\mathbb{R}$ defined over consumption, and a common discount
factor $\beta$. Agent~1 has a random endowment $y_{t}>0$ at date~$t=0,1,\ldots
,$ and agent~2 has a random endowment $Y-y_{t}>0$ where $Y>0$ is a constant
aggregate income. The endowment shock is drawn from a finite set
$\mathcal{Y}\colonequals\{{y}^{1},\ldots,{y}^{n}\}$, $n\geq2$, with $y_{t}$
governed by a Markov chain with stationary transition probabilities
$p(y,y^{+})\colonequals {\mathbf P}(y_{t+1}=y^{+}\mid y_{t}=y)>0$, for all
$y,y^{+}\in\mathcal{Y}$. There is no credit market but agents can transfer
income between themselves at any date. Although \citet{Kocherlakota96} assumes
the endowment shocks are i.i.d., we will show that this convergence result is
easily extended to the case where $y_{t}$ is a Markov chain. It is important to consider this non-i.i.d.\ case. The inter-temporal risk-sharing model with limited commitment has been most frequently applied to village economies where income is predominantly derived from farming. Farm incomes are often found to be to be positively serially correlated.\footnote{For example, \citet{Bold-Broer16} use the ICRISAT data of three Indian villages and find estimated autocorrelation coefficients of around $0.61-0.77$.}

To study optimal risk sharing in this limited commitment context, let $h^{t}=(y_{0},y_{1},\ldots,y_{t})$ denote the history of income realizations,
agents choose a sequence of history-dependent transfers $X_{t}(h^{t})$ from
agent~1 to agent~2 subject to $-Y+y_{t}\leq X_{t}(h^{t})\leq y_{t}$ for each
$h^{t}$ and the self-enforcing constraints that neither agent prefers autarky
from that point on after any history over the agreed transfer plan. In
particular, the self-enforcing constraints for the two agents are
\begin{align*}
u(y_{t}-X_{t}(h^{t}))  &  +\mathbf{E}[\sum_{s=1}^{\infty}\beta
^{s}u(y_{t+s}-X_{t}(h^{t+s}))]\\
&  \geq u(y_{t})+\mathbf{E}[\sum_{s=1}^{\infty}\beta^{s}u(y_{t+s}))],\\
u(Y-y_{t}+X_{t}(h^{t}))  &  +\mathbf{E}[\sum_{s=1}^{\infty}\beta
^{s}u(Y-y_{t+s}+X_{t}(h^{t+s}))]\\
&  \geq u(Y-y_{t})+\mathbf{E}[\sum_{s=1}^{\infty}\beta^{s}u(Y-y_{t+s}))],
\end{align*}
for each date $t$ and $h^{t}$. An \emph{efficient risk-sharing arrangement\/}
will solve (for some feasible $U^{0}$):
\[
\max_{\{X_{t}\}}\mathbf{E}[\sum_{s=0}^{\infty}\beta^{s}u(y_{s}-X_{s}%
(h^{s}))]\quad\mbox{s.t.}\quad\mathbf{E}[\sum_{s=0}^{\infty}\beta
^{s}u(Y-y_{s}+X_{s}(h^{s}))]\geq U^{0}.
\]
and subject to the self-enforcing constraints. It is well known \citep[see,
e.g.,][]{Ligon-Thomas-Worrall02} that the solution at each date has the
following property: For each realization $y$, there is a time-invariant
interval $I_{y}=[\lowerbar{c}_{y},\overline{c}_{y}]$, $\lowerbar{c}_{y}%
\leq\overline{c}_{y}$, such that
\[
c_{t+1}\left(  h^{t+1}\right)  \colonequals y_{t+1}-X_{t+1}(h^{t+1})=%
\begin{cases}
\overline{c}_{y_{t+1}} & \text{if $c_{t}(h^{t})>\overline{c}_{y_{t+1}}$}\\
c_{t}(h^{t}) & \text{if $c_{t}(h^{t})\in I_{y_{t+1}}$}\\
\lowerbar{c}_{y_{t+1}} & \text{if $c_{t}(h^{t})<\lowerbar{c}_{y_{t+1}}$}%
\end{cases}
,
\]
and there is a one-to-one correspondence between feasible $U^{0}$ and
agent~1's initial consumption $c_{0}(h^{0})\in\lbrack\lowerbar{c}_{y_{0}%
},\overline{c}_{y_{0}}]$. We can write this in the
form~\eqref{eq:generalsetting} as $c_{t+1}=f(c_{t},{z}_{t})$ where ${z
}_{t}\colonequals y_{t+1}$, and where
\[
f(c,z)=%
\begin{cases}
\overline{c}_{z} & \text{if $c>\overline{c}_{z}$}\\
c & \text{if $c\in I_{z}$}\\
\lowerbar{c}_{z} & \text{if $c<\lowerbar{c}_{z}$}%
\end{cases}.
\]
The function $f(c,z)$ is clearly monotone increasing in $c$. If $f(c, z)$ were also increasing in $z$ and the Markov process determining $y$ were persistent, then the approach of \citet{HopenhaynPrescott92} could be used. However, even if the Markov process determining $y$ is monotone, the dependence of $f(c,z)$ on $z$ is not easy to derive from the primitives of the model because $\overline{c}_{z}$ and $\lowerbar{c}_{z}$ are computed as part of the optimal solution. They are determined by the slopes of the value functions of the dynamic programming problem and depend on all elements of the problem.\footnote{One case where it is known that monotonicity in $z$ can be established is if one of the agents is risk-neutral. This is the case studied by \citet{Thomas-Worrall88}. We are unaware of any results on the monotonicity in $z$ in other more general cases. Fortunately, our method does not rely on establishing such monotonicity properties and can also be applied if the income process were negatively autocorrelated.}

The first-best risk-sharing allocation is sustainable for some $U^{0}$ if and only if $\cap_{z
}I_{z}\not =\emptyset$. \citet{Kocherlakota96} shows (his Proposition~4.2)
that if shocks are i.i.d.\ and if the first-best is not sustainable then the
distribution of transfers converges weakly to the same non-degenerate
distribution for all $U^{0}$. We now show how to easily extend this result to the case where shocks follow a
Markov chain without making assumptions on the monotonicity of $f(c,z)$ in $z$. Define $c_{\min}\colonequals\min_{z}\overline{c}_{z}$,
$c_{\max}\colonequals\max_{z}\lowerbar{c}_{z}.$ If the first-best is not
sustainable, $\cap_{z}I_{z}=\emptyset$, then $c_{\min}<c_{\max}$. If
$c_{t}\in\lbrack c_{\min},c_{\max}]$, $c_{t+1}=f(c_{t},{z}_{t})\in\lbrack
c_{\min},c_{\max}]$ for all ${z}_{t}$. Define $c\colonequals(c_{\min
}+c_{\max})/2$. Using the notation of Corollary~1 (where $[c_{\min},c_{\max}]$
replaces $[\alpha,\beta]$), let $N_{1}=N_{2}=2$, $z_{1,1}\in\arg\max_{z
}\lowerbar{c}_{z}$, $z_{1,2}\in\arg\min_{z}\overline{c}_{z}$. For any
$z_{0}$, all the assumptions of the corollary are satisfied. Thus, there
exists a unique distribution $\pi$ such that the distributions of $c_{t}$
converge to $\pi$ in the uniform metric for any initial value $c_{0}\in\lbrack
c_{\min},c_{\max}]$. Clearly, $c_{t}\in\lbrack\lowerbar{c}_{z_{0}%
},\overline{c}_{z_{0}}]\cup\lbrack c_{\min},c_{\max}]$ all $t$, and
$[\lowerbar{c}_{z_{0}},\overline{c}_{z_{0}}]\backslash\lbrack c_{\min
},c_{\max}]$ is transient.

If the first-best is sustainable, then the mixing condition is not satisfied.
In that case it can be seen immediately that there is monotone convergence to
a first-best allocation (the limit allocation is dependent on the initial condition).

\section{Conclusion \label{sec:conclusion}}

In this paper we have established convergence results that can be used in a
range of models whose dynamics can be represented by a stochastic recursion,
and which satisfy two main conditions; first, for a given value of the
exogenous driving process, the future value of the endogenous variable is
monotone increasing
in its current value; secondly, the driving process is regenerative. The
latter includes as a special case irreducible finite Markov chains. These two
conditions, along with a standard mixing condition, guarantee weak convergence
to a unique stationary distribution.

This extends the existing results on convergence of monotone Markov processes
that assume the driving process is i.i.d.\ or assume that the driving process
is itself a monotone Markov process \citep{HopenhaynPrescott92}. This
extension is important because most economic models take the driving process
for the underlying shocks to be exogenous and therefore it is useful to have
results for a broader class of stochastic driving processes. Moreover, we do
not require that the stochastic recursion is monotone in the second argument.
This is particularly useful when the stochastic recursion is derived as a
policy function of a dynamic programming problem because establishing
monotonicity in the shock process might require extra restrictions on
preferences or technology.

We have applied our approach to three workhorse models in macroeconomics
extending our understanding of stability in these models. Our Theorem~2 and
its corollary can also be readily used to establish convergence to a unique
stationary distribution for any monotone stochastic recursion in a
regenerative environment where the appropriate mixing condition is satisfied.

\section*{Appendix}

\renewcommand{\theequation}{A.\arabic{equation}} \setcounter{equation}{0}

\begin{arxiv}
\subsection*{Proof of Theorem 1.}
\begin{proof} The metric space of probability distributions on $[a,b]$
with metric $d$ is complete. Due to monotonicity, it is sufficient
to show that
\begin{equation*}\label{conv2}
d(F_t^{(a)}, F_t^{(b)})\to 0
\end{equation*}
exponentially fast.
Then \eqref{conv1} will follow.

Let
$
\varepsilon = \min (\varepsilon_1,\varepsilon_2).
$
Denote by $A$ and $B$ the events
$$
A= \{X_N^{(b)}\le c\} \quad \mbox{and}
\quad
B=\{X_N^{(a)}\ge c \}.
$$
Note that both events are defined by $(\xi_0,\ldots,\xi_{N-1})$,
i.e., belong to the sigma-algebra generated by these
random variables.

{The proof is by induction.} For any $r\ge c$ and for any two probability measures
$\mu$ and $\nu$ on $[a,b]$ with $\mu (x) \equiv \mu [a,x]
\ge \nu (x) \equiv \nu [a,x]$, for all $x$, we may couple initial values of~$4$ trajectories of the Markov chain
$\{X_t^{(b)}\}, \{X_t^{(\nu )}\}, \{X_t^{(\mu )}\}$, $\{X_t^{(a}\}$ in such a way that
$$
1=X_0^{(b)} \ge X_0^{(\nu )} \ge X_0^{(\mu )} \ge X_0^{(a}=0 \quad \mbox{a.s.}
$$
Then
$$
X_t^{(b)} \ge X_t^{(\nu)} \ge X_t^{(\mu)} \ge X_t^{(a)} \quad \text{a.s.\ for any} \quad t,
$$
and we have
\begin{eqnarray*}
0\le F_N^{(\mu )}(r)-F_N^{(\nu )}(r) &=&
{\mathbf P} (X_N^{(\mu )}\le r, A)
+
{\mathbf P} (X_N^{(\mu )}\le r, \overline{A})\\
&-&
{\mathbf P} (X_N^{(\nu )}\le r, A)
-
{\mathbf P} (X_N^{(\nu )}\le r, \overline{A})\\
&=& {\mathbf P} (A)
+
{\mathbf P} (X_N^{(\mu )}\le r, \overline{A})\\
&-&
{\mathbf P} (A)
-
{\mathbf P} (X_N^{(\nu )}\le r, \overline{A})\\
&= &
\int_{\overline{A}} (\mu (S^{(N)}(\overline{v},r))-
\nu (S^{(N)}(\overline{v},r))) {\mathbf P} ((\xi_0,\ldots
\xi_{N-1})\in d \ \overline{v})\\
&\le &
\sup_x (\mu (x) - \nu (x)) \cdot {\mathbf P} (
\overline{A}) \\
&\le &
(1-\varepsilon ) \sup_x (\mu (x) - \nu (x)) .
\end{eqnarray*}

Similarly, for $r < c$, we may use event $B$ to
conclude again that
$$
0\le F_N^{(\mu )}(r)-F_N^{(\nu )}(r)
\le
(1-\varepsilon ) \sup_x (\mu (x) - \nu (x)) .
$$

Therefore,
$$
\sup_r (F_N^{(\mu )}(r)-F_N^{(\nu )}(r))
\le
(1-\varepsilon ) \sup_x (\mu (x) - \nu (x)).
$$

Now, {by induction}, we obtain
$$
0\le F_{kN}^{(a)}(r)-F_{kN}^{(b)}(r) \le (1-\varepsilon )^k
$$
for all $r$.

Indeed, for~$k=1$ the inequality follows from the above. Assume that it holds for~$k\le K-1$. Then
\begin{eqnarray*}
0 & \le & F_{KN}^{(a)}(r)-F_{KN}^{(b)}(r) = {\mathbf P}
\left(X_{KN}^{(a)} \le r \right) - {\mathbf P} \left(X_{KN}^{(b)} \le r \right)
\cr & = & {\mathbf P} \left(X_{N}^{(X_{(K-1)N}^{(a)})} \le r \right) - {\mathbf P}
\left(X_{N}^{(X_{(K-1)N}^{(b)})} \le r \right)
\cr & \le & (1-\varepsilon) \sup_r \left(F_{(K-1)N}^{(a)}(r)-F_{(K-1)N}^{(b)}(r)
\right) \le (1-\varepsilon )^K,
\end{eqnarray*}
which finishes the proof of the induction argument, and the result now follows.
\end{proof}
There is also a straightforward generalization of the above result to the case of a space with a partial order.

\begin{cors}[To Theorem~1] \label{thm:partial} Let $\mathcal S$ be an arbitrary space with a partial order $\le$ such that there exist the least element $s_0 \in \mathcal S$ and the greatest element $s_1 \in \mathcal S$. Assume that a time-homogeneous Markov chain $X_n$ is represented as a stochastic
recursion \eqref{eq:srs} with an i.i.d.\ driving sequence $\{ \xi_n\}$, where function $f:\mathcal{S}\times
{\cal V}\to \mathcal{S}$ is monotone increasing in the first argument (with respect to the partial order $\le$).\\
Assume also that there exist a positive number $\varepsilon$, an integer $N\ge 1$ and sets $\mathcal{C}_u \subset \mathcal{S}$ and $\mathcal{C}_l \subset \mathcal{S}$ such that
\begin{itemize}
\item for every element $s \in \mathcal{S}$, there either exists an element $c \in \mathcal{C}_u$ such that $s \ge c$, or there exists an element $c \in \mathcal{C}_l$ such that $s \le c$;
\item for every $c \in \mathcal{C}_u$,
$$
{\mathbf P}^{(s_1)} (X_N\le c ) > \varepsilon,
$$
and for every $c \in \mathcal{C}_l$,
$$
{\mathbf P}^{(s_0)} (X_N\ge c) > \varepsilon.
$$
\end{itemize}
\end{cors}
The proof follows the lines of proof of the previous theorem.
\end{arxiv}

\subsection*{Proof of Theorem~2}

\begin{proof}
Define a sequence~$Y_{n+1}=X_{T_{n}}$ for all~$n\geq
0$. This sequence is clearly a Markov chain and can
therefore be represented in the form~
\[
Y_{n+1}=g(Y_{n},\eta_{n})
\]
with an i.i.d.\ driving sequence
\[
\eta_{n}=\left(  \tau_{n},Z_{T_{n-1}},..,Z_{T_{n}-1}\right)
\]
and where the function $g$ is defined by
\[
g(Y_{n},\eta_{n})=f^{(\tau_{n})}\left(  Y_{n},Z_{T_{n-1}},..,Z_{T_{n}%
-1}\right)  .
\]
In addition, this recursion is again monotone in the first argument, due to
the monotonicity of function $f$. The assumptions of the theorem imply that
there exists~$c\in\lbrack a,b]$ such that
\[
\mathbf{P}(Y_{1}\leq c|Y_{0}=b)={\mathbf{P}}\left(  \widetilde{X}_{T_{1}%
-T_{0}}^{(b)}\leq c\right)  >0
\]
and
\[
\mathbf{P}(Y_{1}\geq c|Y_{0}=a)={\mathbf{P}}\left(  \widetilde{X}_{T_{1}%
-T_{0}}^{(a)}\geq c\right)  >0.
\]
Hence, the assumptions of Theorem~\ref{BhMa2} are satisfied with the same~$c$
and with~$N=1$. This implies the first statement of the theorem.

We prove the second statement now.
For any $t$, let $\nu(t)$ be such that $T_{\nu(t)}\leq t<T_{\nu(t)+1}$,
so $t$ belongs to the $(\nu(t)+1)$st cycle.
Let $\psi_{t}=(t-T_{\nu(t)},Z_{T_{\nu(t)}},\ldots,Z_{t-1})$ and denote
$\psi_{t,1}=t-T_{\nu(t)}$ and $\psi_{t,2}=(Z_{T_{\nu(t)}},\ldots,Z_{t-1})$, so $\psi_{t}%
=(\psi_{t,1},\psi_{t,2})$.
For any fixed $k>0$ and for all sufficiently large $t$, consider a vector of random
vectors\footnote{Note that each such vector is the sequence of
shocks, together with lengths, of each of the previous $k+1$
completed cycles plus shocks and length of the incomplete cycle up
to time $t$.} $(\eta_{\nu(t)-k},\eta_{\nu(t)-k+1},\ldots,\eta
_{\nu(t)},\psi_{t})$.
By the classical result on regenerative processes (see, e.g., \cite{Asmussen}), for any fixed $k>0$ and as $t$ tends to infinity, the joint distribution of
random vectors $(\eta_{\nu(t)-k},\eta_{\nu(t)-k+1},\ldots,\eta
_{\nu(t)},\psi_{t})$ converges in the total variation norm to the limiting
distribution of a vector of random vectors, say, $(\eta^{-k},\ldots,\eta^{0},\psi^{0})$:
\[
\delta_{t,k}\colonequals\sup_{B}|{\mathbf{P}}((\eta_{\nu(t)-k},\ldots,\eta_{\nu
(t)},\psi_{t})\in B)-{\mathbf{P}}((\eta^{-k},\ldots,\eta^{0},\psi^{0})\in
B)|\rightarrow0
\]
as $t \to \infty$.

Random vectors $\eta^{-k}, \ldots,
\eta^{0},\psi^0$ are mutually independent, and
each of the $\eta^{-j}$, $j=0,\ldots,k$, has the distribution of
the \textquotedblleft typical cycle\textquotedblright, while random vector
$\psi^{0}$ represents the left half of the \textquotedblleft integrated
cycle\textquotedblright, and its first coordinate $\psi_{1}^{0}$ has the
integrated tail distribution ${\mathbf{P}}(\psi_{1}^{0}=l)
=\frac{1}{{\mathbf{E}}\tau_{1}}{\mathbf{P}}(\tau_{1}>l)$, for $l=0,1,\ldots$.
In what follows, we use representation $\psi^{0}=(\psi_{1}^{0},\psi_{2}^{0})$
where $\psi_{2}^{0}$ is the rest of vector $\psi^{0}$ (and, in particular, it
is $l$-dimensional if $\psi_{1}^{0}=l$).

Further, a more advanced construction is possible: one can introduce (on a common probability space with
all earlier defined random variables) a stationary sequence
$(\eta^{-k}_t,\ldots,\eta^{0}_t,\psi^{0}_t)$ such that
%
\[
{\mathbf{P}}(A_{t,k})=\delta_{t,k},%
\]
where we denote
$$A_{t,k} = \{(\eta_{\nu(t)-k},\ldots,\eta_{\nu(t)},\psi_{t})\neq(\eta
^{-k}_t,\ldots,\eta^{0}_t,\psi^{0}_t)\}
$$
(see, e.g., Chapter 1 in \cite{Lindvall}).\footnote{In applied probability,
such a construction is frequently called a ``successful coupling of transient and stationary sequences''.}
In the rest of the proof, we assume such a coupling
to be given.

Introduce $\widehat{Y}^k_{t,0}=\widetilde{Y}^k_{t,0}=Y_{\nu(t)-k}$
and
\[
\widehat{Y}^k_{t,m+1}=g(\widehat{Y}^k_{t,m},\eta_{\nu(t)-k+m}),\quad m=0,..,k-1
\]
and\textbf{\ }
\[
\widetilde{Y}^k_{t,m+1}=g(\widetilde{Y}^k_{t,m},\eta_t^{-k+m}),\quad
m=0,..,k-1.
\]
Consider now
\begin{multline*}
\mathbf{P} \left(\widehat{Y}^k_{t,k} \neq \widetilde{Y}^k_{t,k}\right)
=\mathbf{P} \biggl(g^{(k)}(Y_{\nu(t)-k},(\eta_{\nu(t)-k},\ldots
,\eta_{\nu(t)})) \neq (g^{(k)}(Y_{\nu(t)-k},(\eta_t
^{-k},\ldots,\eta_t^{0}))\biggr)
\\ \leq  \mathbf{P}(\eta_{\nu(t)-k},\ldots,\eta_{\nu(t)})\neq(\eta
^{-k}_t,\ldots,\eta^{0}_t)) \leq \mathbf{P} (A_{t,k}) = \delta_{t,k},
\end{multline*}
with the obvious notation for $g^{(k)}$.

Introduce also $Z_0^k$ as a random variable with distribution $\pi$ and independent of $(\eta^{-k}, \ldots,
\eta^{0},\psi^0)$ and let
$$
Z^k_{t,m+1}=g(Z^k_{t,m},\eta_t^{-k+m}),\quad
m=0,..,k-1.
$$
Note that $Z^k_{t,m}$ has distribution $\pi$ for all $t$, $k$ and $m$.

Due to the first statement of the theorem, we have that, as $k\rightarrow\infty$, the distribution of the random variable
$\widehat{Y}^k_{t,k}$ converges to distribution $\pi$ in the total variation
norm and, hence, in the uniform metric. We can therefore, for any
$\varepsilon>0$, choose $k$ such that, for any $r$,
\[
\left|{\mathbf{P}}(\widehat{Y}^k_{t,k}\leq r)-{\mathbf{P}}(Z^k_{t,k}\leq r)\right|\leq\varepsilon
\]
and then
\begin{multline*}
\left|{\mathbf{P}}\left(\widetilde{Y}^k_{t,k}\leq r \right)-{\mathbf{P}}\left(Z^k_{t,k} \leq r \right)\right|
\\ = \left|{\mathbf{P}}\left(\widetilde{Y}^k_{t,k}\leq r, \widehat{Y}^k_{t,k} = \widetilde{Y}^k_{t,k}\right) + {\mathbf{P}}\left(\widetilde{Y}^k_{t,k}\leq r, \widehat{Y}^k_{t,k} \neq \widetilde{Y}^k_{t,k}\right) -{\mathbf{P}}\left(Z^k_{t,k}\leq r\right)\right| \\
= \left|{\mathbf{P}}\left(\widehat{Y}^k_{t,k}\leq r, \widehat{Y}^k_{t,k} = \widetilde{Y}^k_{t,k} \right) + {\mathbf{P}}\left(\widetilde{Y}^k_{t,k}\leq r, \widehat{Y}^k_{t,k} \neq \widetilde{Y}^k_{t,k} \right) -{\mathbf{P}}\left(Z^k_{t,k}\leq r \right)\right|
\\ = \left|{\mathbf{P}}\left(\widehat{Y}^k_{t,k}\leq r \right) - {\mathbf{P}}\left(\widehat{Y}^k_{t,k}\leq r, \widehat{Y}^k_{t,k} \neq \widetilde{Y}^k_{t,k}\right) + {\mathbf{P}}\left(\widetilde{Y}^k_{t,k}\leq r, \widehat{Y}^k_{t,k} \neq \widetilde{Y}^k_{t,k} \right) -{\mathbf{P}}\left(Z^k_{t,k}\leq r \right)\right|
\\ \leq \left|{\mathbf{P}}\left(\widehat{Y}^k_{t,k}\leq r \right)-{\mathbf{P}}\left(Z^k_{t,k}\leq r \right)\right| + 2\mathbf{P}\left(\widehat{Y}^k_{t,k} \neq \widetilde{Y}^k_{t,k}\right)
\leq 2 \delta_{t,k}+\varepsilon.
\end{multline*}

Now, using similar arguments, for any $r$ and any $l=0,1,\ldots$,
\begin{multline*}
\left|{\mathbf{P}}\left(X_{t}\leq r,t-T_{\nu(t)}=l\right)-{\mathbf{P}}\left(f^{(l)}(\widetilde{X}_{0}^{(\pi)},\psi_{0,2}^{0})\leq r,\psi_{0,1}^{0}=l\right)\right|
\\ = \left|{\mathbf{P}}\left(X_{t}\leq r,t-T_{\nu(t)}=l\right)-{\mathbf{P}}\left(f^{(l)}(Z^k_{t,k},\psi_{t,2}^{0})\leq r,\psi_{t,1}^{0}=l\right)\right|
\\ =\left|{\mathbf{P}} \left(f^{(l)}(\widehat{Y}^k_{t,k},\psi_{t,2})\leq r,t-T_{\nu(t)}=l\right)-{\mathbf{P}%
}\left(f^{(l)}(Z^k_{t,k},\psi_{t,2}^{0})\leq r,\psi_{t,1}^{0}%
=l\right)\right|
\\ \leq \left|{\mathbf{P}} \left(f^{(l)}(\widetilde{Y}^k_{t,k},\psi^0_{t,2})\leq r,\psi^0_{t,1}=l\right)-{\mathbf{P}%
}\left(f^{(l)}(Z^k_{t,k},\psi_{t,2}^{0})\leq r,\psi_{t,1}^{0}%
=l\right)\right| + 2\delta_{t,k}.
\end{multline*}
Note that for any $l=1,2,\ldots$ and any $v\in\mathcal{Z}^{l}$, the set $S_{l}(v,r)=\{x\ :\ f^{(l)}(x,v)\leq r\}$ is
an interval of the form $[a,b)$ or $[a,b]$, for some $b$.
Therefore,
\begin{multline*}
 \left|{\mathbf{P}} \left(f^{(l)}(\widetilde{Y}^k_{t,k},\psi^0_{t,2})\leq r,\psi^0_{t,1}=l\right)-{\mathbf{P}%
}\left(f^{(l)}(Z^k_{t,k},\psi_{t,2}^{0})\leq r,\psi_{t,1}^{0}%
=l\right)\right|
\\ = {\mathbf{P}}(\psi_{t,1}^{0}=l) \int\left|{\mathbf{P}}\left(f^{(l)}(\widetilde{Y}^k_{t,k},v)\leq r\right)-{\mathbf{P}}\left(f^{(l)}(Z^k_{t,k},v)\leq
r\right)\right|{\mathbf{P}}(\psi^{0}_{t,2}\in dv\ |\ \psi_{t,1}^{0}=l)
\\
=\mathbf{P}(\psi_{t,1}^{0}=l) \int\left|{\mathbf{P}}\left(\widetilde{Y}^k_{t,k}\in S_{l}(v,r)\right)-{\mathbf{P}}\left(Z^k_{t,k}\in S_{l}(v,r)\right)\right|{\mathbf{P}}(\psi^{0}_{t,2}\in dv\ |\ \psi_{t,1}^{0}=l)\\
 \leq\mathbf{P}(\psi_{t,1}^{0}=l) \sup_{w}|{\mathbf{P}}(\widetilde{Y}_{t,k}\leq w)-{\mathbf{P}%
}(Z^k_{t,k}\leq w)|.
\end{multline*}
Thus,
\[
\left|{\mathbf{P}}\left(X_{t}\leq r,t-T_{\nu(t)}=l\right)-{\mathbf{P}}\left(f^{(l)}(Z^k_{t,k},\psi_{t,2}^{0})\leq r,\psi_{t,1}^{0}=l\right)\right|
\]
tends to $0$, and the same holds for any finite sum in $l$. From the general theory of renewal processes (see, e.g., \cite{Asmussen}) it is known that the family of random variables $\{t-T_{\nu(t)}\}$ is tight. Recall that this means that
$$\Delta (l):= \sup_t {\mathbf P} (t-T_{\nu (t)} >l )\to 0$$ as $l\to\infty$. Therefore, for any
$\varepsilon >0$, one can choose $L>0$ such that
$\Delta (L) + {\mathbf  P} (\psi^0_{t,1} >L) \le \varepsilon$ for any $t$. Then
\begin{multline*}
\left|{\mathbf P} \left(X_t\le r\right) - {\mathbf P} \left(f^{(l)}(Z^k_{t,k},\psi_{t,2}^{0})\leq r\right)\right|
\\ \le \sum_{l=0}^L \left|{\mathbf P} \left(X_t\le r,t-T_{\nu(t)}=l\right) - {\mathbf P} \left(f^{(l)}(Z^k_{t,k},\psi_{t,2}^{0})\leq r\psi_{t,1}^{0}=l\right) \right| +\varepsilon \to \varepsilon,
\end{multline*}
as $t\to\infty$. Letting $\varepsilon$ go to zero, we arrive at the second statement of the theorem.

The proof of the convergence of $(X_t,Z_t)$ follows the exact same lines, with an extra event added in each of the probabilities. We omit this derivation as the formulae are rather cumbersome but do not contain any additional technical difficulties.
\end{proof}

\subsection*{Proof of Corollary~1}

\begin{proof}
We have to show that Corollary \ref{Cor1} follows from Theorem
\ref{thm:general}. For that, we have to define a typical (say, first)
regenerative cycle and show that all the conditions of Theorem
\ref{thm:general} hold. Assume that $Z_{0}=z_{0}$, so $T_{0}=0$. Let
$T_{1}=\tau_{1}=\min\{t>0:Z_{t}=z_{0}\}$, then the aperiodicity means that
$G.C.D.\{t:{\mathbf{P}}(T_{1}=t)>0\}=1$. Let $T_{n}=\sum_{1}^{n}\tau_{j}$
where $\tau_{j}$ are i.i.d.\ copies of $\tau_{1}$. Let the conditions of the
Corollary hold, and $k_{i}$ be the number of occurrences of $z_{0}$ in
the sequence $z_{j,i}$, for $i=1,2$. Let $L$ be the \emph{least common
multiple\/} of $k_{1}$ and $k_{2}$,
\[
L=\min\{l\ :\ l/k_{1}\ \ \mbox{and}\ \ l/k_{2}\ \ \mbox{are integers}\}.
\]
Let $\alpha$ be a random variable that takes values $0$ and $1$ with equal
probabilities and does not depend on any of the processes defined in the
model. Then define a regenerative cycle as follows: $\widehat{T}_{0}=0$ and
\[
\widehat{T}_{1}=T_{1}\alpha+T_{L}(1-\alpha).
\]
That is, we suppose that our regenerative cycle is either a single cycle or a
sum of $L$ cycles, with equal probabilities. Then all the conditions of
Theorem \ref{thm:general} hold (with $\widehat{T}_{i}$ in place of $T_{i}$).
Indeed, condition \eqref{eq:finitemean} follows since it holds for $\tau_{1}$,
and since $\widehat{T}_{1}$ is not bigger than $T_{L}$, the sum of $L$ copies
of $\tau_{1}$. Condition \eqref{eq:aperiodicity} follows because the set of
all $t$ such that ${\mathbf{P}}(\widehat{T}_{1}=t)>0$ includes the set of all
$t$ such that ${\mathbf{P}}(\tau_{1}=t)>0$ and, therefore,
\[
G.C.D.\{t\ :\ {\mathbf{P}}(\widehat{T}_{1}=t)>0\}\leq
G.C.D.\{t\ :\ {\mathbf{P}}(\tau_{1}=t)>0\},
\]
so, {given aperiodicity, both greatest common divisors are equal to $1$.}
Finally, $\varepsilon_{1}$ in \eqref{eq:top} is not smaller than $\frac{1}%
{2}p_{1}\delta_{1}>0$ and, similarly, $\varepsilon_{2}$ in \eqref{eq:bottom}
is not smaller than $\frac{1}{2}p_{2}\delta_{2}>0$.
\end{proof}

\subsection*{Proof of Lemma 1.}

\begin{proof}
Define
\[
\hat{c}:=(\bar{e}-\lowerbar{e})/(1-(\beta R)^{1/\gamma}).
\]
Clearly, there exists $\hat{x}$ such that for $x>\hat{x}$, $c\left(  x,e\right)
>\hat{c}$ for all $e\in E$.\footnote{For $a\geq(R/(R-1))(1-\beta
)^{1/(1-\gamma)}\hat{c}$ setting $c_{t}=((R-1)/R)a+e_{t}$ all $t$ (so that
$a_{t}$ is constant at $a$) yields a discounted utility greater than $\hat
{c}^{1-\gamma}/\left(  1-\gamma\right)  ;$ this is higher utility than any
policy with $c(a,e_{t})\leq\hat{c}$ which yields at most $\hat{c}^{1-\gamma
}/\left(  1-\gamma\right)  $.} Suppose that, at some $(x,e)$ with $x>\hat{x}$,
$f(x,e)\geq x.$ We demonstrate a contradiction. Since $f(x,e)>\lowerbar{x},$
the Euler condition holds with equality:%
\begin{equation}
u^{\prime}(c(x,e))=\beta R\mathbf{E}\left[  u^{\prime}(c(f(x,e),e^{+}%
))\mid e\right]  .\label{eqn:eulerr3}%
\end{equation}
(\ref{eqn:eulerr3}) implies that there exists $X^{+}\in\mathcal{E}$ with
$Q\left(  e,X^{+}\right)  >0$ and such that $u^{\prime}(c(x,e))\leq\beta
Ru^{\prime}(c(f(x,e),e^{+}))$ for $e^{+}\in X^{+}$. Thus for $e^{+}\in X^{+},$%
\[
c(f(x,e),e^{+})^{-\gamma}\geq(\beta R)^{-1}c(x,e)^{-\gamma},
\]
so%
\begin{equation}
c(f(x,e),e^{+})\leq(\beta R)^{1/\gamma}c(x,e).\label{eqn:88}%
\end{equation}
By $c(x,e)>$ $\hat{c},$ we have from (\ref{eqn:88}):
\begin{align}
c(x,e)-c(f(x,e),e^{+}) &  \geq(1-(\beta R)^{1/\gamma})c(x,e)\\
&  >(\bar{e}-\lowerbar{e}).\label{eqn:10}%
\end{align}
Then%
\begin{align}
f(f(x,e),e^{+}) &  =R(f(x,e)+e^{+}-c(f(x,e),e^{+}))\nonumber\\
&  >R(x+e^{+}+(\bar{e}-\lowerbar{e})-c(x,e))\nonumber\\
&  \geq R(x+e-c(x,e))\nonumber\\
&  =f(x,e),\label{eqn:12}%
\end{align}
where the first line follows from the budget constraint, the second from
$f(x,e)\geq x$ and (\ref{eqn:10}), the third from $e^{+}\geq\lowerbar{e}$ and
$\bar{e}\geq e$, and the last from the budget constraint. Defining $x_{t}=x$,
$x_{t+1}=f(x,e)$, $x_{t+2}=f(f(x,e),e^{+})$ etc., we can express
(\ref{eqn:12}) as $x_{t+2}\geq x_{t+1}.$ Repeating the logic of (\ref{eqn:88})
and (\ref{eqn:12}), starting at $(f(x,e),e^{+})$ for some $e^{+}\in X^{+}$
there is some $X^{++}\in\mathcal{E}$ with $Q\left(  e^{+},X^{++}\right)  >0$
at $t+2$ such that $x_{t+3}>x_{t+2}$ and such that
\[
c_{t+2}\leq(\beta R)^{2/\gamma}c(x,e),
\]
etc. Iterating, we get eventually that $c_{t+n}<\hat{c}$ while $x_{t+n}%
>\hat{x}$, a contradiction.
\end{proof}

\subsection{Details of convergence result in Section \ref{subsec:huggett}}

Assume the initial state (at time $t=0$) $e_{0}$ is the atom of the chain and
suppose that $x_{0}=\hat{x}$. Maximum consumption at $t=0$ if all resources
are used is $\overline{c}:=\hat{x}-\underaccent{\bar}{x}/R+e_{0}.$ We can also
define a lower bound on consumption at any date by $\underaccent{\bar}{c}>0$%
.\footnote{Since $c_{t}=((R-1)/R)\underaccent{\bar}{x}+\bar{e}$ $>0$ is always
feasible (by assumption on $\underaccent{\bar}{x}$), this implies a lower
bound to utility; consumption below some positive level implies a discounted
utility below this bound.} Choose $T\in\{1,2,\ldots\}$ and $\xi>0$ so that
\begin{equation}
\overline{c}^{-\gamma}>(\beta R)^{T}\underaccent{\bar}{c}^{-\gamma}%
+\xi.\label{eqn:er}%
\end{equation}
(This implies that the agent would like, if feasible, to transfer a small
amount of consumption forward from $T$ periods ahead.) Suppose that $\Pr
[x_{t}=\underaccent{\bar}{x}]=0$ for all $t>0$. We shall establish a
contradiction. For any $\Delta>0,$ we can choose $\varepsilon>0$ so that
${\mathbf{P}}(x_{t}<\underaccent{\bar}{x}+\varepsilon
\ \mbox{for at least one}\ t\in\{1,\ldots,T\})<\Delta$ (using the right
continuity of the distribution of $x_{t}$, say $F_{t}$, with the hypothesis
that $F_{t}(\underaccent{\bar}{x})=0$ for $t\leq T,$ choose $\varepsilon$ so
that at each $t,$ $F_{t}(\underaccent{\bar}{x}+\varepsilon)<\Delta/T$ ). It
follows that an increase in consumption at $t=0$ of amount $\lambda
\leq\varepsilon R^{-T}$ can be financed by a reduction at date $T$ (but
otherwise keeping time $t$ consumption $c_{t},$ $1\leq t<T,$ at its original
level),\thinspace\ i.e., $x_{t}\geq\underaccent{\bar}{x}$ for $1\leq t\leq T,$
with probability at least $(1-\Delta)$ since assets at $t$ would be
$x_{t}-R^{t}\lambda\geq x_{t}-\varepsilon\geq\underaccent{\bar}{x}$ for $t\leq
T$. To a first-order, the discounted utility cost is at most $\lambda(\beta
R)^{T}\underaccent{\bar}{c}^{-\gamma}$. Otherwise reduce $c_{t}$ to restore
assets to $x_{t}$ when the credit constraint first binds at $t<T,$ at a cost
of at most $\lambda\underaccent{\bar}{c}^{-\gamma}$. The change in utility to
a first order is thus at least%
\[
\lambda\left(  \overline{c}^{-\gamma}-\Delta\underaccent{\bar}{c}^{-\gamma
}-(1-\Delta)(\beta R)^{T}\underaccent{\bar}{c}^{-\gamma}\right)  .
\]
Choosing $\Delta$ small so that $\Delta\underaccent{\bar}{c}^{-\gamma}<\xi,$
the term multiplying $\lambda$ is positive, using (\ref{eqn:er}), and so for
$\lambda$ small (so that $\lambda\leq\varepsilon R^{-T}$ is satisfied, where
$\varepsilon$ depends on $\Delta,$ and that higher order terms are small
enough) there is a profitable deviation. Hence
$\hat{t}:=\min
\{ t>0 :
{\mathbf P}
(x_t=\lowerbar{x})>0\} <\infty$.

Next, define times as $T_{0}=\min\left\{  t\geq0:e_{t}=e_{0}\right\}  $, and
for $j=0,1,\ldots,$
\[
T_{j+1}=\min\left\{  t\geq T_{j}+\hat{t}:e_{t}=e_{0}\right\}  .
\]
Thus the sequence $\left\{  e_{t}\right\}  $ with associated times $\left\{
T_{n}\right\}  $ is regenerative and satisfies
\eqref{eq:finitemean}-\eqref{eq:aperiodicity}. Moreover consider the process
$\tilde{x}_{t}^{(\alpha)}$starting at $t=0$ from $\alpha$ and satisfying
recursion $\tilde{x}_{t+1}^{(\alpha)}=f(\tilde{x}_{t}^{(\alpha)},e_{T_{0}+t}%
)$. By the above, $\{\tilde{x}_{\hat{t}}^{(\hat{x})}$ $=$
$\underaccent{\bar}{x}\}$ has positive probability. Now consider $\tilde
{x}_{T_{1}-T_{0}}^{(\underaccent{\bar}{x})}$. By the monotonicity of $f$ in
its first argument, $\tilde{x}_{t}^{(\underaccent{\bar}{x})}\leq\tilde{x}%
_{t}^{(\hat{x})}$ for all $t$, and if $\tilde{x}_{\hat{t}}^{(\hat{x}%
)}=\underaccent{\bar}{x}$, then also $\tilde{x}_{\hat{t}}%
^{(\underaccent{\bar}{x})}=\underaccent{\bar}{x}$, so $\tilde{x}%
_{t}^{(\underaccent{\bar}{x})}=\tilde{x}_{t}^{(\hat{x})}$ for $t\geq\hat{t}$
and conditional on hitting $\underaccent{\bar}{x}$ at $\hat{t}$, $\tilde
{x}_{\tilde{t}}^{(\hat{x})}$ and $\tilde{x}_{\tilde{t}}%
^{(\underaccent{\bar}{x})}$ coincide at each $t\geq\hat{t}$. Consequently $c$
exists satisfying conditions \eqref{eq:top}-\eqref{eq:bottom} of Theorem 2 and
the result follows.

\subsection*{Proof of Lemma 2.}

\begin{proof}
(i) Suppose that $\lowerbar{g}(k)\leq k$. Consider $z_{\tau}$,
$k_{\tau}$ such that $k_{\tau+1}=\lowerbar{g}(k_{\tau})=g(k_{\tau
},z_{\tau})$, that is consider the shock that depletes capital at the maximum
rate. Let $\phi(k)=\inf_{z}f_{k}(k,z)$ be the greatest lower bound on the
marginal product as a function of $k$. We have
\begin{align*}
u^{\prime}(c(k_{\tau},z_{\tau}))  &  =\beta\mathbf{E}\left[
u^{\prime}(c(g(k_{\tau},z_{\tau}),z_{\tau+1}))f_{k}(g(k_{\tau},z_{\tau
}),z_{\tau+1})\mid z_{\tau}\right]%
\\
&  =\beta\mathbf{E}\left[  u^{\prime}(c(\lowerbar{g}(k_{\tau}),z_{\tau+1}))f_{k}%
(\lowerbar{g}(k_{\tau}),z_{\tau+1})\mid z_{\tau}\right]\\
&  \geq\beta\phi(k_{\tau})\mathbf{E}\left[  u^{\prime
}(c(\lowerbar{g}(k_{\tau}),z_{\tau+1}))\mid z_{\tau}\right] \\
&  \geq\beta\phi(k_{\tau})\mathbf{E}\left[  u^{\prime}(c(k_{\tau
},z_{\tau+1}))\mid z_{\tau}\right].%
\end{align*}
The first equality follows by equation~\eqref{rameuler}. The second equality follows by the definition $k_{\tau+1}=\lowerbar{g}(k_{\tau})$. The inequality in the third line follows by $\lowerbar{g}(k_{\tau})\leq k_{\tau}$ and the definition of
$\phi$, and the final inequality follows by $\lowerbar{g}(k_{\tau})\leq k_{\tau}$ and $c(k,z)$ increasing in $k$.
Since $z_{\tau}=z^{i}$ for some state~$i$, the above inequality (deleting
terms for states~$j\not =i$) implies
\[
u^{\prime}(c(k_{\tau},z^{i}))\geq\beta\phi(k_{\tau})u^{\prime}(c(k_{\tau
},z^{i}))p(z^{i},z^{i}).
\]
Since $u^{\prime}(c)>0$, it therefore follows that $1\geq\beta\phi(k_{\tau
})p(z^{i},z^{i})$. Let $\rho\colonequals\min_{i}p(z^{i},z^{i})$. By assumption
$\rho>0$, and therefore $\phi(k_{\tau})\leq1/(\beta\rho)$ for all $k_{\tau}$.
Since $\beta>0$ and $\rho>0$, equivalently, $k_{\tau}\geq\phi^{-1}%
(1/(\beta\rho))$. Letting $\epsilon=\phi^{-1}(1/(\beta\rho))$, the assumption
that $f_{k}(k,z)\rightarrow\infty$ for all $z$ as
$k\mathrel{\nonscript\mkern-1.2mu\mkern1.2mu{\downarrow}}0$ implies
$\epsilon>0$ and hence we have $k_{\tau}\geq\epsilon$ for all $\tau$. Thus, it
follows that $g(k_{\tau},z_{\tau})>k_{\tau}$ for all $k_{\tau}<\epsilon$ and
all $z_{\tau}$.
(ii) Suppose not. Then by continuity of $\lowerbar{g}(k),$ $\exists$
$\hat{k}>k^{\prime}$ such that $\lowerbar{g}(\hat{k})=\hat{k}$. By
definition of $k^{\prime}$ and assumption that $k^{\prime}<k^{\prime\prime}$,
$\lowerbar{g}(k^{\prime\prime})<\bar{g}(k^{\prime\prime})$
($=k^{\prime\prime}$) and so $\hat{k}>k^{\prime\prime}$.
Consider any $(k,z)\in$ $[k^{\prime},k^{\prime\prime}]\times
\widehat{\mathcal{Z}}$. Then $g(k,z)\in\lbrack k^{\prime},k^{\prime\prime}]$
since $g(k,z)\geq\lowerbar{g}(k)\geq\lowerbar{g}(k^{\prime
})=k^{\prime}$ where the second inequality follows from $g$ increasing in $k$,
and the equality from the definition of $k^{\prime}$; likewise, $g(k,z)\leq
\bar{g}(k)\leq\bar{g}(k^{\prime\prime})=k^{\prime\prime} $ where the second
inequality follows from $g$ increasing in $k$, and the equality from the
definition of $k^{\prime\prime}$.
Similarly, for $k\geq\hat{k}$, $g(k,z)\geq\hat{k}$, $\forall z\in
\widehat{\mathcal{Z}}$, since $g(k,z)\geq\lowerbar{g}(k)\geq
\lowerbar{g}(\hat{k})=\hat{k}$.
We shall demonstrate a contradiction. Take any $(\bar{k},z_{0})\in$
$(k^{\prime},k^{\prime\prime})\times\widehat{\mathcal{Z}}$, and define
recursively
\begin{align}
\bar{k}_{0}  &  =\bar{k};\nonumber\\
\bar{k}_{\tau}  &  =g(\bar{k}_{\tau-1},z_{\tau-1})\qquad\tau=1,\ldots,N.
\label{defkbar}%
\end{align}
Iterating (\ref{vkeuler}) $N>0$ times:%
\begin{equation}
v_{k}(\bar{k},z)=\mathbf{E}\left[  \beta^{N}\Pi_{\tau=0}^{N-1}f_{k}(\bar
{k}_{\tau},z_{\tau})v_{k}(\bar{k}_{N},z_{N})\mid z_{0}\right]  .
\label{iterate}%
\end{equation}
Likewise, for any $\tilde{k}\geq\hat{k}$, defining $\tilde{k}_{\tau}$
(analogously to $\bar{k}_{\tau}$) starting from $( \tilde{k},z_{0})$,%
\begin{equation}
v_{k}(\tilde{k},z)=\mathbf{E}\left[  \beta^{N}\Pi_{\tau=0}^{N-1}f_{k}%
(\tilde{k}_{\tau},z_{\tau})v_{k}(\tilde{k}_{N},z_{N})\mid z_{0}\right]  .
\label{iterate1}%
\end{equation}
By $\bar{k}_{\tau}\in[ k^{\prime},k^{\prime\prime}] $, $\tilde{k}_{\tau}%
\geq\hat{k}$, $k^{\prime\prime}<\hat{k}$, and the the strict concavity of $f$
and $v$ in $k$ :
\begin{equation}
f_{k}(\bar{k}_{\tau},z_{\tau})\geq\gamma f_{k}(\tilde{k}_{\tau},z_{\tau
})\qquad a.s., \label{rank1}%
\end{equation}
for some $\gamma>1$, and
\begin{equation}
v_{k}(\bar{k}_{N},z_{N})>v_{k}(\tilde{k}_{N},z_{N})\qquad a.s. \label{rank2}%
\end{equation}
Thus, from (\ref{iterate}), (\ref{iterate1}), (\ref{rank1}) and (\ref{rank2}):%
\[
v_{k}(\bar{k},z)>\gamma^{N}v_{k}(\tilde{k};z).
\]
Since $v_{k}(\bar{k};z)<\infty$ by $\bar{k}>0$, $\gamma>1$, $v_{k}(\tilde
{k};z)>0$, letting $N\rightarrow\infty$ yields a contradiction.
\end{proof}

\noindent\textbf{Acknowledgement:} We thank two referees and the Associate Editor for their helpful and constructive comments. The research of S.~Foss was partially
supported by EPSRC grant EP/I017054/1 and by RSF grant 17-11-01173. The research of S.~Shneer was
supported by EPSRC grant EP/L026767/1. The research of J.~Thomas and
T.~Worrall was supported by ESRC grant ES/L009633/1.




\end{document}